\documentclass{article}
\usepackage{geometry}
\usepackage[english]{babel}
\usepackage[utf8]{inputenc}
\usepackage[T1]{fontenc}
\usepackage{graphicx,color}
\usepackage[final]{pdfpages}
\usepackage[extra]{tipa}
\usepackage{titlesec}
\usepackage{xspace}		
\usepackage[svgnames]{xcolor}
 \usepackage{lmodern}
\usepackage[hypertexnames=false]{hyperref}

\usepackage{mathtools}
\usepackage{graphicx}
\usepackage[ruled,vlined]{algorithm2e}
\usepackage{amsmath,amsfonts}
\usepackage{dsfont}
\usepackage{amssymb}
\usepackage[ntheorem]{empheq}
\usepackage{mathenv}
\usepackage[standard,thref,amsmath,thmmarks]{ntheorem}
\usepackage{nicefrac}

\usepackage{theoremref}

\usepackage{caption,subcaption}
\usepackage{color}
\usepackage{comment}
\usepackage{amsfonts}
\usepackage{bm}

\usepackage{minitoc}
\usepackage{cite}

\numberwithin{counter}{subsection}

\usepackage{aliascnt}

\renewtheorem{theorem}{Theorem}

\newaliascnt{lem}{theorem}
\renewtheorem{lemma}[lem]{Lemma}
\aliascntresetthe{lem}

\newaliascnt{def}{theorem}
\renewtheorem{definition}[def]{Definition}
\aliascntresetthe{def}

\newaliascnt{prop}{theorem}
\renewtheorem{proposition}[prop]{Proposition}
\aliascntresetthe{prop}

\newaliascnt{coro}{theorem}
\renewtheorem{corollary}[coro]{Corollary}
\aliascntresetthe{coro}

\newaliascnt{rmk}{theorem}
\renewtheorem{remark}[rmk]{Remark}
\aliascntresetthe{rmk}

\newtheorem*{property*}{Property:}

\newtheorem*{question*}{Question:}

\newtheorem*{assumptions*}{Assumptions:}

\newcommand{\D}{\partial}

\newcommand{\R}{\mathbb{R}}

\newcommand{\dt}{{\Delta t}}
\newcommand{\dx}{{\Delta x}}

\renewcommand{\P}{\mathbb{P}}
\renewcommand{\(}{ \left(}
\renewcommand{\)}{ \right)}
\newcommand\independent{\protect\mathpalette{\protect\independenT}{\perp}}
\def\independenT#1#2{\mathrel{\rlap{$#1#2$}\mkern2mu{#1#2}}}

\definecolor{Question}{HTML}{F77409}
\definecolor{Problem}{HTML}{F01A1A}
\definecolor{Comment}{HTML}{011AF8}
\definecolor{Check}{HTML}{08AB06}

\usepackage{authblk}

\title{Exact description of SIR-Bass epidemics on 1D lattices}
\author[1]{Gadi Fibich\thanks{fibich@tau.ac.il}}
\author[1]{Samuel Nordmann\thanks{samnordmann@gmail.com}}
\affil[1]{Department of Applied Mathematics, Tel Aviv University}

\begin{document}

\maketitle


\noindent {\bf MSC2020 Class. No:} 92D30, 90B60

\begin{abstract}

This paper is devoted to the study of a stochastic epidemiological model which is a variant of the SIR model to which we add an extra factor in the transition rate from susceptible to infected accounting for the inflow of infection due to immigration or environmental sources of infection. This factor yields the formation of new clusters of infections, without having to specify a priori and explicitly their date and place of appearance.

We establish an {exact deterministic description} for such stochastic processes on 1D lattices (finite lines, semi-infinite lines, infinite lines) by showing that the probability of infection at a given point in space and time can be obtained as the solution of a deterministic ODE system on the lattice. Our results allow stochastic initial conditions and arbitrary spatio-temporal heterogeneities on the parameters.

We then apply our results to some concrete situations and obtain useful qualitative results and explicit formulae on the macroscopic dynamics and also the local temporal behavior of each individual. In particular, we provide a fine analysis of some aspects of cluster formation through the study of {patient-zero problems} and the effects of {time-varying point sources}.

Finally, we show that the space-discrete model gives rise to new space-continuous models, which are either ODEs or PDEs, depending on the rescaling regime assumed on the parameters.

\end{abstract}

\tableofcontents

\section{Introduction}
\subsection{Goal of the paper}

This paper is devoted to the study of a stochastic epidemic model on 1D (directed or undirected) lattices with a general epidemic source term and arbitrary heterogeneity on the parameters and initial conditions.
Our model is a variant of the stochastic SIR model, originally introduced in its deterministic version by Kermack and McKendrick~\cite{Kermacka}, which forms the framework of almost any model used in epidemiology. Despite the extensive literature on the SIR model (see, e.g., the books of Bailey~\cite{Bailey1975b}, Anderson and May~\cite{Anderson1991}, Mollison~\cite{Mollison1995} and the review of Heathcote~\cite{H2000a}), some fundamental and practical questions remain open. In the current context of the propagation of the COVID-19, an in-depth study of this model and its variants is a major challenge.

In addition to the usual mechanism of contagion occurring in epidemics, we also assume the presence of an epidemiological \emph{source} term, qualified as ``\emph{external influences}''. This term acts as an additional mechanism that leads to the transition from susceptible to infected, which does not depend on the level of infection in the population, and accounts for the inflow of infection due to immigration or environmental sources of infection. In our framework, it is allowed to depend on space and time. From a modeling point of view, dealing with external influences is important since they yield the formation of new clusters of infections, without having to specify a priori and explicitly their date and place of appearance.
Such external influences are not traditionally included in the $SIR$ model and refer more to the Bass model~\cite{Bass1969}, originally introduced to describe the diffusion of new products.
In the marketing context, {external influences} account for the spontaneous adoptions of the product induced by advertising or mass media, while ``social imitation" and ``word of mouth'' are formally regarded as a ``contagion'' mechanism. 
In turn, the population of \emph{potential adopters} of a product (resp. \emph{adopters}) are seen as formal analogs of the population of \emph{susceptible} (resp. \emph{infected}) in the epidemiology context. Although each modeling situation has its own specificities, we keep a unified approach to both phenomena.
In what follows, we use more frequently the terminology borrowed from epidemiology, but our considerations can be transposed to the context of the adoption of a new product. Let us mention that the \emph{recovery} of infected individuals operating in the $SIR$ model is also relevant in the marketing context even though it has not been included in the Bass model until recently~\cite{Fibich2016}.
Similar models are also used in other social situations: propagation of riots~\cite{Berestycki2015,Epstein2002}, rumours~\cite{DALEY1964,Dietz1967}, etc.

\paragraph*{}

The main goal of the present paper is to establish an \emph{exact deterministic description} for such stochastic SIR-Bass epidemics on 1D lattices. Specifically, we show that the probability of infection at a given point in space and time can be obtained as the solution of a deterministic ODE system on the lattice. We then apply our results to some concrete situations and obtain useful qualitative results and explicit formulae. In particular, we provide a fine analysis of some aspects of cluster formation through the study of \emph{patient-zero problems} and the effects of \emph{time-varying point sources}.

For simplicity and a clear presentation of our results and methodology, we focus in this paper on 1D lattices (i.e., finite, semi-infinite, and infinite lines) where the infection only occurs from a node to its direct neighbors.
Our approach can be adapted without difficulties to the case where the contact network is a tree (i.e., a graph with no cycles). Some of our ideas may also apply to a more general context, which will be the subject of future works.

As a counterpart to the strong assumption on the structure of the graph, we try to keep the assumptions on the parameters and initial conditions as broad as possible by allowing any type of spatial-temporal heterogeneity of the coefficients. Understanding the effect of such local heterogeneities, even for simple networks, is still a major challenge in epidemiology~\cite[Section~1]{Pellis2015b}.

We emphasize that our results provide an exact description of the process at the individual-scale, that is, we describe both the macroscopic dynamics of the epidemics and the local temporal behavior of each individual.
This takes us one step further than the usual description at the population-scale through aggregate quantities such as the expected fraction of infected individuals in the population. Such a spatial-temporal description at the local scale is required for predicting population-level dynamics from individual-level observations, all the more in the situation of strong heterogeneities.

\paragraph*{Outline.} We give a brief state of the art of the mathematical epidemiology literature in Section~\ref{sec:context}, which allow us to present the context of our study, to motivate it, and to highlight the novelty of our results.
Our mathematical framework and its link with classical models are presented in Sections~\ref{sec:Framework} and~\ref{sec:Non-spatial}.

We first present and prove our results dealing with 1D one-sided lattices in which the contagion can only occur in one direction.
We focus in Section~\ref{sec:1D1S} on the Bass model when there is no recovery. Section~\ref{sec:1D1S_SIR} deals with the full SIR-Bass model.
Our main results (\autoref{th:1D1SWithP}, \autoref{th:1D1SWithP_R}) establish that the probability of a certain node to be susceptible at some time satisfies a differential equation which is continuous in time and discrete in space. This provides an exact spatial-temporal deterministic description of the stochastic process, for any type of heterogeneity on the parameters and stochastic initial conditions.
We then use these results to perform a further analysis and derive explicit formulae in several typical situations, such as the spatially homogeneous case (\autoref{rmk:1D1S_NonSpatialSolution}, \autoref{th:Coro_Homogeneous_1D1S_General}), the patient-zero problem (\autoref{th:coro_patient_zero}, \autoref{th:coro_patient_zero_SIR}), and the case of a time-varying point sources (\autoref{th:SourceTerm}, \autoref{th:SourceTerm_R}). We also propose a further analysis on the speed of propagation of epidemics (\autoref{th:SpeedPropagation1D1S}) as well as a detailed comparison of our model with the classical aggregate $SIR$ model (Section~\ref{sec:Comparison_SIR}).

In Section~\ref{sec:1D2S}, we extend our results to two-sided 1D lattices (corresponding to an undirected underlying graph) in which epidemics can propagate in both directions. We derive a new useful formula that expresses the probability of infection in a two-sided lattice in terms of the probabilities of infections in one-sided lattices. 
In Section~\ref{sec:1D1S_SpaceContinuous}, we discuss two different space rescalings giving rise to either a limiting space-continuous ODE or a limiting PDE coupling time and space.
Finally, we give some concluding remarks in Section~\ref{sec:conclusion}.


 %
%

\subsection{Context and motivation}\label{sec:context}

To present the context of our work, we provide a brief overview of the mathematical study of epidemiological models.
Most epidemiological models proposed in the literature fall into two broad categories: a deterministic approach using either ODEs~\cite{Kendall1965,Hethcote1989} or PDEs of the reaction-diffusion type~\cite{DiekmannBIS,Diekmann1978,Thieme1977,Ruan}, and a stochastic approach using individual-centered Poisson processes (agent-based models)~\cite{Bailey1975b,Anderson1991}. 
The deterministic and stochastic approaches have mostly developed in parallel in the literature.
In general, deterministic models can be obtained as the ``mean-field'' limits~\cite{Kurtz1970,Simon2005} of stochastic models. 
Mean-field models describe the average dynamics and typically apply to sufficiently large populations where stochastic fluctuations and individual variabilities are negligible~\cite{Kurtz1971}. 
They have the advantage of being more amenable to analysis and to give a clear idea of some qualitative aspects of epidemic dynamics. Deterministic models are particularly useful to study the spatial speed of propagation and the shape of the propagation front of epidemics~\cite{Hosono1995,Diekmann1979,Aronsonb}. On the other hand, stochastic models are capable of reproducing observed data way more accurately, both quantitatively and qualitatively, see, e.g.,~\cite[Section~5.1 \& 7]{Bailey1975b} and~\cite{Mollison1995}. The main reasons for this are, firstly, that stochasticity may have a dramatic influence on the overall dynamics since small populations are involved (for example during the outbreak of an epidemic) and, secondly, that the stochastic framework is readily well suited for the incorporation of strong heterogeneities or complex social structures which are known to be of tremendous importance in epidemics propagation.

The simplest models, referred to as \emph{aggregate} or \emph{compartmental} models, assume that the population is homogeneously mixed~\cite[Section~5.10]{Bailey1975b}, that is, the connectivity between two individuals is constant among all pairs of individuals.
In this situation, individuals are interchangeable and so the behavior of the system can be described by aggregate quantities (such as the expected fraction of infected individuals).
The classical analysis of these models relies on strong symmetry assumptions on the parameters and initial conditions~\cite{Kurtz1970,Kurtz1971}.

However, many studies reveal that homogeneous mixing only holds in small groups but not in large populations, see e.g.~\cite[Section~5.10]{Bailey1975b} and~\cite{Keeling2005a}.
Intuitively, at the scale of a city, it does not make sense to assume that an individual interacts as much with members of his family as with a socially unrelated individual living on the other side of the city. 
The limited number of interactions of each individual tends to hinder the spread of infection and to cause the formation of localized clusters. For these reasons, the network within which social interactions occur has a dramatic impact on the epidemic propagation, and so aggregate models turn out to be inaccurate to describe the spatial-temporal dynamics of the epidemics.

A more elaborate approach is to consider that each individual interacts only with a relatively small subset of the population.
The connectivity between individuals is usually encoded by a weighted graph, called \emph{contact network} in the epidemiology context~\cite{Keeling1999,Andersson2000}. We point out that this formalism assumes that social interactions are fixed once and for all, which is not a problem when considering a sufficiently short epidemic episode (evolving networks have also been proposed, see~\cite{Enright2018a,Jiang2019} and references therein).
Contact networks are typically sparse and may embed various social variables such as place of residence, age, profession, etc.
The different classes of contact networks considered in the literature can be roughly classified according to their \emph{clustering} level~\cite{Badham2010,Keeling2005b} reflecting the extent to which the system is capable of generating and maintaining localized clusters of infections. For example, \emph{lattice networks}~\cite{Sat61994,Sazonov2008,Matsuda1992} regard the population as placed on a regular grid of points with spatially localized connections, yielding a high clustering level; on the contrary, \emph{random networks}~\cite{Decreusefond2012,Ball2012,Andersson1998} have a low clustering level since connections between individuals occur randomly regardless of their locations.
Lattices and random networks appear as two opposite extreme instances, each one focusing on certain aspects of the dynamics observed in epidemics and discarding others. Small-world or scale-free networks have been proposed as examples of more realistic configurations, combining both localized and rare long-range connections, which display intermediate clustering levels. For more details, we refer the reader to the reviews of Keeling \& Eames~\cite{Keeling2005a}, Britton~\cite{Britton2010} and Danon et. al.~\cite{Keeling2011}. Let us emphasize here that lattices, especially 2D lattices, are the physically relevant contact networks when considering contagion phenomena that primarily depend on the physical distance between the infected and susceptible individuals, which happens to be the case in the propagation of some diseases such as Measles~\cite{Grenfell2001} or the diffusion of some products such as solar panels~\cite{Graziano2015,Bollinger2012}.

Most of the approaches developed for studying complex networks make strong assumptions of homogeneity and symmetry at the individual-scale and are mainly concerned with aggregate quantities (such as the total number of infected) which describe the dynamics at the population level but fail to precisely grasp the dynamics at the individual-scale. It remains nevertheless a major challenge to understand the local behavior and the effects of strong heterogeneity~\cite[Section~1]{Pellis2015b}.
For example, the study of the \emph{patient-zero} problem (where the epidemics spread from a single infected individual at initial time) or the influence of a spatially localized source of infection gives significant insights into the spatial propagation of epidemic clusters.

Several approximation methods have been proposed to deal with complex networks with a high level of heterogeneity. The goal of these approaches is to provide deterministic approximations that go beyond the insufficient mean-field approximation, but that are still less computationally demanding and more amenable to analysis than the full stochastic models.
The pair-approximation method~\cite{Rand2009,Keeling1999} consists in formally neglecting the possible intercorrelation between the nodes induced by cycles in the contact network. Under this formal approximation, the dynamics can be described by a system of ODEs featuring aggregate quantities such as the average number of neighbors per node and the ratio of triples of nodes which form an interconnected triangle (this quantity reflects the level of clustering of the contact network). 
While this method offers great flexibility and results which are quite in agreement with numerical simulations of the stochastic process, its theoretical ground is far from being completely understood.
Other approximations have been proposed~\cite{House,Bansal2007}, such as degree-based models, probability generating function (PGF) formalism, edge-based compartment modeling, but all these methods can be derived from the pair-approximation method~\cite{House}.

Regarding theoretical results and exact deterministic descriptions of the stochastic process, most of the studies, if not all, assume a tree-like structure on the contact network at some level.
If the contact network is a tree-graph (i.e. a loopless graph), it is well known that the pair-wise approximation is exact~\cite{Sharkey2015} and yields an exact description of the stochastic process, namely, the probabilities of infection satisfy a system of ODEs. This allows a fine theoretical analysis of the dynamics.
Results are also available when the graph has a local tree-like structure (such as random \emph{configuration} networks~\cite{Keeling2011}) at the limit as the population size becomes large and under homogeneity assumptions~\cite{Decreusefond2012,Ball2013,Keeling2011}.

As mentioned in the introduction, the present paper deals with 1D lattices, which are particular cases of tree-graphs, and for which exact equations are available.
The novelty of our approach is in the combination of the following ingredients. Firstly, we study the SIR-Bass model, that is, we add to the classical SIR model an external source term accounting for a spontaneous inflow of infection; secondly we allow any heterogeneity on the coefficients and allow stochastic initial conditions; thirdly, we use the obtained exact deterministic description to perform a further analysis at a local scale and derive explicit formulae on some specific typical epidemiological instances; finally we show that the space-discrete model gives rise to new space-continuous models, which are either ODEs or PDEs, depending on the rescaling regime assumed on the parameters.

\section{Framework}\label{sec:Framework}
In this paper, we consider time-discrete stochastic processes that converge to time-continuous processes as the time step tends to zero. Discrete and continuous models each have their own advantages: the former can be simulated numerically directly, while the latter is more amenable for analysis.


\subsection*{Time-discrete process}
Let us fix a time-step $0<\dt\ll1$ and consider random processes represented by a family of random variables $\left((x_k^n)_{k\in\mathcal{K}}\right)_{n\in\mathbb{N}}$, where $x_k^n$ represents the \emph{state} of individual $k\in \mathcal{K}$ at time $t^n:=n\dt$. The set $\mathcal{K}$ can be finite $\mathcal{K}=\{1,\dots,K\}$, infinite $\mathcal{K}=\mathbb{Z}^d$, or semi-infinite $\mathcal{K}=\mathbb{N}$.
The state of individual/node $k$ at time $t^n$ can be \emph{susceptible} $(x_k^n=s)$, \emph{infected} $(x_k^n=i)$, or \emph{recovered} $(x_k^n=r)$.
We define the following events:
$$
\begin{gathered}
{S}_k^n\text{ is the event that }x_k^n=s,
\\
 {I}_k^n\text{ is the event that }x_k^n=i,
 \\
  {R}_k^n\text{ is the event that }x_k^n=r,
\end{gathered}
$$
and denote their probabilities by
\begin{equation*}
{}[S_k^n]=\P(S_k^n),\qquad [I_k^n]=\P(I_k^n),\qquad [R_k^n]=\P(R_k^n).
\end{equation*}
The events $S_k^n$, $I_k^n$, $R_k^n$ are disjoint and complementary, and so 
\begin{equation}\label{S+I+R=1}
{}[S_k^n]+[I_k^n]+ [R_k^n]\equiv1,\qquad \forall n\geq 0,\ k\in\mathcal{K}.
\end{equation}

We assume that the stochastic dynamics in the time-interval $(t^n,t^{n+1})=(n\dt,(n+1)\dt)$ is governed by the following mechanisms:
\begin{itemize}
\item If node $k$ is susceptible at $t^n$, it can become infected
\begin{itemize}
\item[-] by a source term (``external influence'') at a rate of $p_k^n$,
\item[-] and by contagion from infected individual $i$ (``internal influence'') at a rate of $q_{ik}^n$.
\end{itemize}
These effects are additive.
\item If a node $k$ is infected at $t^n$, it can recover at a rate of $r_k^n$.
\item If node $k$ is recovered at $t^n$, it remains so for all later times.
\end{itemize}
The transition probabilities are thus given by
\begin{subequations}\label{Assumption_Generale_Intro_BIS}
\begin{equation*}\label{Assumption_Generale_1_Intro_BIS}
 \tag*{(\ref{Assumption_Generale_Intro_BIS}a)}
\P\left(I_k^{n+1}\big\vert\bm{X}^n \right)=
\left\{\begin{aligned}
&\dt\left(p_k^n+ \sum_{i\in \mathcal{K}}q_{ik}^n \mathds{1}_{I_{i}^n}\right),&&\quad\text{if }x_k^n=s,\\
& 1-r_k^n\dt,&&\quad\text{if }x_k^n=i,\\
& 0 ,&&\quad\text{if }x_k^n=r,
\end{aligned}\right.
\end{equation*}
where $\bm{X}^n=(x_k^n)_{k\in\mathcal{K}}$ is the state of the network at time $t^n$,
$$\mathds{1}_{I_{i}^n}=
\begin{cases}
1, &\quad if $x_i^n=i,$\\
0, &\quad else,
\end{cases}$$
and by
\begin{equation*}\label{Assumption_Generale_2_Intro_BIS}
 \tag*{(\ref{Assumption_Generale_Intro_BIS}b)}
\P\left(R_k^{n+1}\big\vert \bm{X}^n \right)=
\left\{\begin{aligned}
&0,&&\quad\text{if }x_k^n=s,\\
& r_k^n\dt,&&\quad\text{if }x_k^n=i,\\
& 1 ,&&\quad\text{if }x_k^n=r.
\end{aligned}\right.
\end{equation*}
\end{subequations}
The parameters $p_k^n$, $q_{ik}^n$, $r_k^n$ are nonnegative and may depend on space $k$ and time $n$.
We also assume that
$\bm{X}^{n+1}$ is a family of random variables that are mutually independent when conditioned with respect to $X^n$. 
This way, the process is well defined.
The matrix $(q_{ik})_{\mathcal{K}}$ represents the weighted directed edges of the contact graph between the nodes.



The state variables $(x_k^n)$ are coupled through the sum $\sum_{i\in \mathcal{K}}q_{ik}^n\mathds{1}_{I_{i}^n}$, see~\ref{Assumption_Generale_1_Intro_BIS}. 
This term, which embeds the structure of the graph, has a central role in the analysis.
As noted previously, in this paper we focus on the case of 1D (directed or undirected) lattices, where $q_{ik}^n\equiv0$ when $i\neq k\pm1$.

\subsection*{Initial conditions}


The starting point of the stochastic dynamics~\eqref{Assumption_Generale_Intro_BIS} is the initial conditions $$(x_k^0)_{k\in\mathcal{K}},\qquad x_k^0:=x_k(t=0),$$ which are assumed to be uncorrelated random variables taking values in $\{s,i,r\}.$ Thus, in terms of the probabilities $[S_k^n]$, $[I_k^n]$ and $[R_k^n]$, we consider the initial conditions
\begin{equation}\label{hyp:initial_cond_uncor}
\left\{\begin{aligned}
&[S_k^0],[I_k^0],[R_k^0]\in[0,1],&&\quad k\in\mathcal{K},\\
&[S_k^0]+[I_k^0]+[R_k^0]\equiv 1,&&\quad k\in\mathcal{K},\\
&x^0_k\independent x^0_l&&\quad k\neq l\in\mathcal{K}.
\end{aligned}\right.
\end{equation}

In some applications, it is natural to consider \emph{deterministic} (\emph{pure}) initial conditions, where the initial state of each node is given in a deterministic way. 
This is a special case of~\eqref{hyp:initial_cond_uncor}, when $x_k^0$ is determined with probability $1$ among the possible states $\{s,i,r\}$.
For example, in the $SIR$ model, it is standard to assume that the epidemics start from ``patients zero'', and so that initially, all nodes are susceptible, except for a few which are infected.
In the Bass model, the standard assumption is that the whole population is susceptible (non-adopters) when the new product is first introduced into the market.

In this paper, we analyze the more general case~\eqref{hyp:initial_cond_uncor} where the initial state of each node is given by a random variable.
In addition to being more general from a mathematical point of view, dealing with stochastic initial conditions is important from an application point of view. 
For example, in the context of epidemiology, it may be that the only available information is the percentage of the population in each area which are infected. Similarly, in the context of the diffusion of a new product, sometimes only the fraction of the population that has adopted the product in each area is available, rather than a precise description of the state of each individual.



\subsection*{Time-continuous limit}
As noted, the $n$ superscript denotes a state variable evaluated at the discrete time $t^n=n\dt$. Throughout the paper, we assume the following:
\paragraph*{Time-continuous limit assumptions.}
\textit{
 The parameters $p_k^n$, $q_k^n$ and $r_k^n$ in~\eqref{Assumption_Generale_Intro_BIS} converge to some functions $p_k(t)$, $q_k(t)$ and $r_k(t)$ respectively, as $\dt\to0$ and $n\to+\infty$ so that $t=n\dt$ is constant. 
}
\paragraph*{}
Then, we classically have that the time-discrete process $(x_k^n)$ converges to a time-continuous Poisson-type process $(x_k(t))$.
We use the same notations as in the time-discrete setting, by replacing the dependence on $n$ with a dependence on $t$. Thus,
\begin{gather*}
{S}_k(t)\text{ is the event that }x_k(t)=s,
\\
 {I}_k(t)\text{ is the event that }x_k(t)=i,
 \\
  {R}_k(t)\text{ is the event that }x_k(t)=r,
\end{gather*}
and
\begin{equation*}
[S_k](t)=\P(S_k(t)), \qquad [I_k](t)=\P(S_k(t)),\qquad [R_k](t)=\P(S_k(t)).
\end{equation*}
The time-continuous process is a Poisson-type process, defined as follows. The transition of individual $k$ from \emph{susceptible} to \emph{infected} at time $t$ occurs through an exponential law with the rate
$
 \lambda^{s\to i}_k(t)= p_k(t)+\sum_{i\in\mathcal{K}}q_{ik}(t)\mathds{1}_{I_{i}(t)},
$
and from \emph{infected} to \emph{recovered} with the rate
$
 \lambda^{i\to r}_k(t)=r_k(t).
$

\section{Aggregate models}\label{sec:Non-spatial}

%

The aggregate version of the above model~\eqref{Assumption_Generale_Intro_BIS} is obtained when the parameters and initial conditions do not depend on $k$, which corresponds to the case where the underlying graph is homogeneous and complete. 
In this case, in the limit as the population size becomes infinite, the dynamics is governed by the aggregate SIR-Bass model~\cite{Fibich2016} which is given by the system of ODEs
\begin{equation}\label{SIR-Bass}
\left\{\begin{aligned}
&S'(t)=-S(p+qI),\\
&I'(t)=S(p+qI)-rI,\\
&R'(t)=rI,
\end{aligned}\right.
\end{equation}
where $S$, $I$, and $R$ represent the fraction of susceptible, infected, and recovered, respectively (or equivalentely the probability that any node is susceptible, infected, or recovered respectively). Similarily to~\eqref{S+I+R=1}, we have that
\begin{equation}\label{S+I+R=1_aggregate}
S(t)+I(t)+R(t)\equiv 1,\qquad t\geq 0.
\end{equation}

If we take $p=0$ in~\eqref{SIR-Bass}, i.e., we assume that there is no source term, then the SIR-Bass model reduces to the classical aggregate $SIR$ model (see e.g.~\cite{Hethcote1989})
\begin{equation}\label{ClassicalSIR}
\left\{\begin{aligned}
&S'(t)=-qSI,\\
&I'(t)=qSI-rI,\\
&R'(t)=rI.
\end{aligned}\right.
\end{equation}
To have a non-trivial dynamics in~\eqref{ClassicalSIR}, we need to assume the presence of patient(s) zero, i.e., $I(0)>0$.

Alternatively, if we take $r=0$ and $R(t=0)=0$ (i.e. assume that there is no recovery), then $R(t)\equiv0$ and so by~\eqref{S+I+R=1_aggregate} we have $I=1-S$. Hence,~\eqref{SIR-Bass} reduces to the original Bass model~\cite{Bass1969}:
\begin{equation}
I'(t)=(1-I)(p+qI).
\end{equation}
If we further assume that $I(0)=0$, this leads to the Bass formula
\begin{equation}\label{Bass_formula}
I_{Bass}(t)=\frac{1-e^{-(p+q)t}}{1+\frac{q}{p}e^{-(p+q)t}}.
\end{equation}

Finally, assuming no recovery ($r=R(t=0)=0$) and no source term ($p=0$), the SIR-Bass model~\eqref{SIR-Bass} reduces to the classical $SI$ model
\begin{equation}\label{SI}
\left\{\begin{aligned}
&S'(t)=-qSI,\\
&I'(t)=qSI.
\end{aligned}\right.
\end{equation}

\section{Bass model on 1D one-sided lattices}\label{sec:1D1S}


We begin our analysis with the case of a 1D one-sided lattice, that is, when
\begin{equation}\label{Assumption1D1S}
q_{ik}=0\qquad \text{if }i\neq k-1.
\end{equation}
Therefore, infection of node $k$ at time-step $n$ can be caused either by a source term at the rate $p_k^n$, or by contagion from its left neighbor at the rate $q_k^n$.
Hence,~\ref{Assumption_Generale_1_Intro_BIS} reduces to\footnote{In the case of the semi-infinite line $\mathcal{K}=\mathbb{N}$, the transition probability for the boundary node $k=0$ is
\begin{equation}
 \P\left(I_0^{n+1}\big\vert \bm{X}^n \right)=
\left\{\begin{aligned}
&p_0^n\dt ,&&\text{if }x_0^n=s,\\
& 1-r_0^n\dt,&&\text{if }x_0^n=i,\\
& 0 ,&&\text{if }x_0^n=r,
\end{aligned}\right.
\end{equation}
}
\begin{subequations}\label{Assumption_1D1S_WithRecovery_Intro}
\begin{equation}\label{Assumption_1D1S_WithRecovery_1_Intro}
 \tag*{(\ref{Assumption_1D1S_WithRecovery_Intro}a)}
 \P\left(I_k^{n+1}\big\vert \bm{X}^n \right)=
\left\{\begin{aligned}
&\left(p_k^n+q_k^n \mathds{1}_{I_{k-1}^n}\right)\dt,&&\text{if }x_k^n=s,\\
& 1-r_k^n\dt,&&\text{if }x_k^n=i,\\
& 0 ,&&\text{if }x_k^n=r,
\end{aligned}\right.
\end{equation}
and~\ref{Assumption_Generale_2_Intro_BIS} remains unchanged, that is:
\begin{equation}\label{Assumption_1D1S_WithRecovery_2_Intro}
 \tag*{(\ref{Assumption_1D1S_WithRecovery_Intro}b)}
\P\left(R_k^{n+1}\big\vert \bm{X}^n \right)=
\left\{\begin{aligned}
&0,&&\text{if }x_k^n=s,\\
& r_k^n\dt,&&\text{if }x_k^n=i,\\
& 1 ,&&\text{if }x_k^n=r,
\end{aligned}\right.
\end{equation}
\end{subequations}

In this section, we focus on the Bass model case when there is no recovery, i.e., when
\begin{equation}\label{1D1S_Assumption_NoRecovery}
r_k^n\equiv 0\qquad \text{ and }\qquad [R^0_k]\equiv 0,\quad\forall k\in\mathcal{K}.
\end{equation}
This implies that $[R_k^n]\equiv 0$ and so by~\eqref{S+I+R=1},
\begin{equation}\label{S+I=1}
{}[S_k^n]+[I_k^n]\equiv 1,\qquad \forall k,n.
\end{equation}
Therefore,~\eqref{Assumption_1D1S_WithRecovery_1_Intro} and~\eqref{Assumption_1D1S_WithRecovery_2_Intro} reduce to
\begin{equation}\label{Assumption_1D1S_WithoutRecovery}
\P\left(I_k^{n+1}\big\vert \bm{X}^n \right)=
\left\{\begin{aligned}
&\dt\left(p_k^n+q_k^n \mathds{1}_{I_{k-1}^n}\right),&&\text{if }x_k^n=s,\\
& 1,&&\text{if }x_k^n=i.
\end{aligned}\right.
\end{equation}

\subsection{Deterministic description}
In what follows, we show that the probability $[S_k^n](t)$ of node $k$ to be susceptible at time $t^n$ satisfies a deterministic discrete PDE {(here, the term ``discrete'' means that the time and space derivatives are replaced by finite differences)}.
We then let $\dt\to0$ to deduce that the probabilities $[S_k](t)$ satisfy a system of differential equations.
\begin{theorem}\label{th:1D1SWithP}
Assume that the nodes are placed on a 1D one-sided graph, see~\eqref{Assumption1D1S}, that the initial conditions are stochastic and uncorrelated, see~\eqref{hyp:initial_cond_uncor}, and that there is no recovery, see~\eqref{1D1S_Assumption_NoRecovery}, so that the stochastic dynamics are governed by~\eqref{Assumption_1D1S_WithoutRecovery} {where $\mathcal{K}$ is either a finite, semi-infinite or infinite line.}
Then for all $k\in\mathcal{K}$, as $\dt\to0$, $[S_k]$ satisfies the differential equation
\begin{equation}\label{EquationDiscrete1D1SidedWithP_Heterogeneous}
{}[S_k]'(t)+(p_k(t)+q_k(t))[S_k](t) -q_k(t)e^{-\int_0^t p_k(s)ds}[S_k^0][S_{k-1}](t)=0,\qquad [S_k](t=0)=[S_k^0].
\end{equation}
In particular, if $p\equiv p_k^n$ and $q\equiv p_k^n$ are independent of $k$ and $n$, then~\eqref{EquationDiscrete1D1SidedWithP_Heterogeneous} reads
\begin{equation}\label{EquationDiscrete1D1SidedWithP}
{}[S_k]'+(p+q)[S_k] -qe^{-pt}[S_k^0][S_{k-1}]=0 ,\qquad [S_k](t=0)=[S_k^0].
\end{equation}
\end{theorem}

\begin{proof}
A straightforward computation gives
\begin{align*}
\P(S_k^n)-\P(S_k^{n+1})
&=\P(I_k^{n+1}\cap S^n_k)=\P(I_k^{n+1}\cap S^n_k\cap S^n_{k-1})+\P(I_k^{n+1}\cap S^n_k\cap I^n_{k-1})\\
&=\underbrace{\P(I_k^{n+1}\vert  S^n_k\cap S^n_{k-1})}_{=p_k^n\dt}\P( S^n_k\cap S^n_{k-1})+\underbrace{\P(I_k^{n+1}\vert  S^n_k\cap I^n_{k-1})}_{=(p_k^n+q_k^n)\dt}\P(S^n_k\cap I^n_{k-1}).
\end{align*}
Since $\P(S_k^n\cap I_{k-1}^n)=\P(S_k^n)-\P(S_k^n\cap S_{k-1}^n)$, we obtain that
\begin{equation}\label{Proof_1D1S_noR_1}
\frac{\P(S_k^{n+1})-\P(S_k^n)}{\dt}+ (p_k^n+q_k^n)\P(S_k^n)-q_k^n\P\(S_k^n\cap S_{k-1}^n\)=0.
\end{equation}

In principle, we now have to deal with the term $\P(S_k^n\cap S_{k-1}^n)$ which involves the 2-marginals of the process through the joint laws of $x_{k-1}^n$ and $x_{k}^n$. The 2-marginals, however, can be expressed through the 1-marginal $([S_k])_{k\in\mathcal{K}}$ as follows.
Let us write
\begin{equation}\label{Closure1DSteps}
\P(S_k^n\cap S_{k-1}^n)
= \P(S_k^n\cap S_k^0\cap S_{k-1}^n)
= \P(S_k^n\vert S_k^0\cap S_{k-1}^n)\P(S_k^0\cap S_{k-1}^n)
\end{equation}
and compute each term of the product separately. {
First, since the initial conditions are uncorrelated, 
the variables $x_k^0$ and $x_{k-1}^n$ are independent, hence
\begin{equation}\label{1D1S_Indep_Proof}
\P(S_k^0\cap S_{k-1}^n)
=\P(S_k^0)\P(S_{k-1}^n).
\end{equation}
This identity follows from the fact that there is no path from $k$ to $k-1$, therefore the computation of the probability $\P(S_{k-1}^n)$ through~\eqref{Assumption_1D1S_WithoutRecovery} does not involve $x_k^0$.
Second,
if $k$ is not infected at initial time and $k-1$ is not infected by $t^n$, then $k$ can only become infected through the source term $p_k$. In this case, node $k$ remains susceptible during the time interval $(t^n,t^{n+1})$ with a probability $1-p_k^n$, i.e.,
\begin{equation}\label{FromProductToExp}
 \P(S_k^n\vert S_k^0\cap S_{k-1}^n)=\prod\limits_{i=0}^{n-1}\(1-p_k^i\dt\)=e^{-\sum_{i=0}^{n-1}p_k^i\dt}(1+O(\dt))=e^{-\int_{0}^{t}p_k(\cdot)}(1+O(\dt)).
\end{equation}
Note by the way that identities~\eqref{1D1S_Indep_Proof}-\eqref{FromProductToExp} does not hold if the graph is a circle (i.e. $\mathcal{K}=\mathbb{Z}/K\mathbb{Z}$) since, roughly speaking, node $k-1$ can become infected by an epidemic starting from $k$.
}
%
Substituting \eqref{FromProductToExp} and~\eqref{1D1S_Indep_Proof} in~\eqref{Closure1DSteps}, we derive the \emph{closure identity}
\begin{equation}\label{1D1S_P_Closure}
\P(S_k^n\cap S_{k-1}^n)= e^{-\int_0^{t^n} p_k(s)ds}\P(S_k^0)\P(S_{k-1}^n)(1+O(\dt)).
\end{equation}
Finally, plugging~\eqref{1D1S_P_Closure} into~\eqref{Proof_1D1S_noR_1}, one gets
\begin{equation}\label{Proof_1D1S_noR_2}
\frac{\P(S_k^{n+1})-\P(S_k^n)}{\dt}+ (p_k^n+q_k^n)\P(S_k^n)-q_k^ne^{-\int_0^{t^n} p_k(s)ds}\P(S_k^0)\P(S_{k-1}^n)=O(\dt).
\end{equation}
Letting $\dt\to0$ gives the result.
\end{proof}
\begin{remark}
The key element in the proof is the closure identity~\eqref{1D1S_P_Closure}. This identity holds pointwise whenever a node is only influenced by a single node.
Therefore, even if we drop the assumption that the graph is 1D and let it be general, relation~\eqref{EquationDiscrete1D1SidedWithP_Heterogeneous} still holds pointwise at any node which is only influenced by a single node.
\end{remark}
\begin{remark}
The initial conditions $[S_k^0]$ explicitely appear in ODE~\eqref{EquationDiscrete1D1SidedWithP_Heterogeneous} for $[S_k]$, and not just in the initial conditions. As mentioned already, the case of deterministic initial conditions is a special case where $[S_k^0]\in\{0,1\}$.
\end{remark}

\autoref{th:1D1SWithP} provides a spatial-temporal deterministic description of the stochastic dynamics~\eqref{Assumption_1D1S_WithoutRecovery}.
Once compute the solution $[S_k](t)$ through equation~\eqref{EquationDiscrete1D1SidedWithP_Heterogeneous}, we can complete the description of the stochastic process, as follows: $[I_k]$ is deduced from~\eqref{S+I=1}, the second-order marginals can be computed using the closure identity~\eqref{1D1S_P_Closure}, and any higher-order marginal can be computed using the same methodology.

We emphasize that equation~\eqref{EquationDiscrete1D1SidedWithP_Heterogeneous} on $[S_k](t)$ is exact and does not result from any approximation (such as the mean-field approximation) or from the assumption that the population is infinite.

%
%
%
%

\subsection{Explicit solutions}
In the spatially-homogeneous case, one can obtain explicit solutions of~\eqref{EquationDiscrete1D1SidedWithP_Heterogeneous}:
\begin{proposition}[Spatially-homogeneous Bass solution]\label{rmk:1D1S_NonSpatialSolution}
Assume the conditions of \autoref{th:1D1SWithP}. If the initial condition $[S_k^0]\equiv [S^0]$ and the parameters $p_k(t)\equiv p(t),q_k(t)\equiv q(t)$ do not depend on space $k$ (we allow, however, $p$ and $q$ to depend on time $t$), then $[S_k](t)=[S](t)$ does not depend on $k$, and is given by
\begin{equation}\label{ExplicitU_Hetero_T}
{}[S](t)=[S^0]e^{-\int_0^t \left(p(s)+q(s)\right)ds+[S^0]\int_0^tq(s)e^{-\int_0^s p}ds }.
\end{equation}
\end{proposition}
\begin{proof}
By translation symmetry, $[S_k](t)$ does not depend on $k$. Hence, equation~\eqref{EquationDiscrete1D1SidedWithP_Heterogeneous} reads
\begin{equation*}
{}[S]'+\left(p(t)+q(t)\left(1-[S^0]e^{-\int_0^t p(s)ds}\right)\right)[S] =0,\qquad [S](t=0)=[S^0].
\end{equation*}
Expression~\eqref{ExplicitU_Hetero_T} follows from a straightforward integration of the above equation.
\end{proof}

If, in addition to the assumptions of~\autoref{rmk:1D1S_NonSpatialSolution} , $p$ and $q$ do not depend on $t$, then
\begin{equation}\label{ExplicitU}
{}[S](t)=
\left\{\begin{aligned}
&[S^0]e^{-(p+q)t+q[S^0]\frac{1-e^{-pt}}{p}}, &&\quad\text{if }p>0,\\
&[S^0]e^{-q(1-[S^0])t}, &&\quad\text{if }p=0.
\end{aligned}\right.
\end{equation}
This expression was already obtained in~\cite{Fibich2010a} for $p>0$ and $[S^0]\equiv1$.

%

\paragraph*{}
Another situation where an explicit solution can be calculated is that of a semi-infinite line when the boundary node is infected at $t=0$ and all other nodes are initially susceptible:

\begin{proposition}[Patient-zero Bass problem]\label{th:coro_patient_zero}
Assume the conditions of \autoref{th:1D1SWithP}, let $p_k(t)\equiv p$ and $q_k(t)\equiv q$ be independent of $k$ and $t$, and let the nodes be placed on the semi-inifinite line $\mathcal{K}=\{0,1,2,\dots\}$, such that
\begin{equation*}
x_0^0=i,\qquad x_k^0=s,\quad k=1,2,\dots
 \end{equation*} 
Then the solution $[S_k](t)$ of~\eqref{EquationDiscrete1D1SidedWithP} is given by
\begin{equation}\label{Explicit_withP_Uk(t)}
{}[S_0](t)\equiv0,\qquad [S_k](t)=e^{-(p+q)t}\sum\limits_{l=0}^{k-1} \frac{\left(q \frac{1-e^{-pt}}{p}\right)^l}{l!},\qquad k=1,2,\dots
\end{equation}
In addition, if we denote by $N_K(t)=\sum_{k=1}^K \mathds{1}_{x_k(t)=i}$ the total number of infected nodes among nodes $\{1,\dots,K\}$, then
\begin{equation}\label{Explicit_Esperance_1D1S}
\mathbb{E}[N_K^n]=K-e^{-(p+q)t} \sum\limits_{l=1}^{K} l\frac{\left(q \frac{1-e^{-pt}}{p}\right)^{K-l}}{(K-l)!}.
\end{equation}
\end{proposition}
\begin{proof}
Expression~\eqref{Explicit_withP_Uk(t)} can be verified by direct substitution in~\eqref{EquationDiscrete1D1SidedWithP}. Expression~\eqref{Explicit_Esperance_1D1S}, follows from $\mathbb{E}[N_K^n]= \sum_{k=1}^{K}[I_k](t)$, \eqref{S+I=1}, \eqref{Explicit_withP_Uk(t)} and some elementary computations detailed in Appendix~\ref{app:proof_line}
\end{proof}

Let us briefly justify the importance of the patient-zero Bass problem considered in~\autoref{th:coro_patient_zero}.
When the source term $p$ is much smaller than the contagion term $q$, the dynamics of the population can be formally described in two time scales: propagation of clusters through contagion in a fast time-scale, and spontaneous creation of new seeds in a slow time-scale~\cite{Fibich2010a}.
The patient-zero Bass problem thus models the fast time-scale evolution of clusters.

Passing to the limit in~\eqref{Explicit_withP_Uk(t)}, we have that
\begin{equation}\label{Explicit_withP_Uk(t)_limit}
\lim_{k\to+\infty} [S_k](t)=e^{-(p+q)t+q\frac{1-e^{-pt}}{p}},\qquad k=1,2,\dots
\end{equation}
This limit coincides with the expression for $[S_k]$ in the homogeneous case with $[S_k^0]\equiv1$, see~\eqref{ExplicitU}, since the effect of patient zero vanishes as $k\to+\infty$.
Note also that, as $K\to+\infty$, the expected fraction of infected nodes is given by, see~\eqref{Explicit_Esperance_1D1S},
\begin{equation}\label{ExpectedFraction}
\lim\limits_{K\to+\infty}\frac{\mathbb{E}[N_K^n]}{K}= 1-e^{-(p+q)t+q \frac{1-e^{-pt}}{p}}
\end{equation}
which expresses again the fact that the solution of the patient zero problem converges to the homogeneous solution given by~\eqref{ExplicitU} as $K\to+\infty$.

\paragraph*{}
In~\autoref{th:coro_patient_zero}, the epidemic is triggered by two distinct mechanisms: firstly, by contagion from ``patient-zero'' located at $k=0$, secondly, by spontaneous infection from the source term $p$.
Expression~\eqref{Explicit_withP_Uk(t)} indicates how these two effects combine nonlinearly.
To focus on the effect of the source term, we may remove ``patient zero'' and assume that all the individuals on the half-line are initially susceptible. In this case, it can be shown that the probability $[S_k](t)$ coincides with the probability of a node to be susceptible on a finite circle of length $k$. Therefore, it can be computed explicitly, see~\cite{Fibich2019} for more details.
If, on the other hand, we only focus on the effect of contagion from ``patient-zero'', we may discard the source term by letting $p\to0$ in~\eqref{Explicit_withP_Uk(t)}. This gives the solution of the \emph{patient-zero $SI$ problem} 
\begin{equation}\label{LabelQuelconque}
{}[S_k](t)= e^{-qt}\sum_{l=0}^{k-1}\frac{\left(q t\right)^{l-1}}{(l-1)!},\qquad k=1,2,\dots
\end{equation}
By~\eqref{LabelQuelconque}, for any $k\geq1$ and $t>0$, we have on the one hand that $S_k(t)>0$, implying that the epidemics propagates instantaneously from ``patient zero'' to any other node. On the other hand, $\lim_{k\to+\infty} S_k(t)=1$ for any $t>0$ implying that the propagation occurs at a finite \emph{asymptotic speed}. In fact, we can show that the asymptotic speed of propagation is $q$:
\begin{proposition}[Speed of propagation]\label{th:SpeedPropagation1D1S}
Assume the conditions of~\autoref{th:coro_patient_zero} and assume furthermore that $p= 0$ ($SI$ patient-zero problem). Denote by
$\overline k(t)=\sup\{k\geq0: x_k(t)= i\}$ the location of the propagation front of the epidemic. Then, for all $\alpha\in\R$, 
\begin{equation*}
 \lim\limits_{t\to+\infty}\P\(\overline{k}(t)\leq qt+ \alpha\sqrt{t}\)=\frac{1}{\sqrt{2\pi q}}\int_{-\infty}^\alpha e^{-\frac{\tau ^2}{2q}}d\tau.
 \end{equation*} 
\end{proposition}
\begin{proof}
The process $\overline{k}(t)$ is equivalent to a standard Poisson process with parameter $q$. It is then classical that the mean value of $\overline{k}(t)$ is $t$. Then, by the central limit theorem, we have that $\frac{\overline{k}(t)-t}{\sqrt{t}}$ converges in probability when $t\to+\infty$ to a Gaussian with mean $0$ and variance $q\sqrt{t}$, hence the result.
\end{proof}
Thus the epidemics propagate through space at an average speed $q$, such that the exact location of the front has a variance $q\sqrt{t}$ around its mean value.
Therefore, we can make a formal analogy between the propagation of the epidemics and the propagation of particles whose density $u(t,x)$ is given by the advection-diffusion equation~\cite{Socolofsky2004}
\begin{equation*}
\D_t u+q\D_x u-\frac{q}{2}\D^2_{xx} u=0.
\end{equation*}
This equation describes the evolution of a density of particles subject to advection towards $x>0$ at speed $q$ (resulting in propagation at average speed $q$) and to diffusion (i.e., each particle follows a Brownian motion) with diffusivity $\frac{q}{2}$ (resulting in a variance $q$ for the location of the propagation front).

\paragraph*{}

Finally, we compute an explicit solution for a time-varying source term located at the boundary of a semi-infinite line:
\begin{proposition}[Time-varying point source]\label{th:SourceTerm}
Assume that the individuals are placed on a semi infinite line $\mathcal{K}=\{0,1,2,\dots\}$ and that all individuals are initially susceptible, i.e., $x_k^0=s$ for $k\in\mathcal{K}$. Assume 
that $q_k(t)\equiv q>0$ is constant, that $p_k(t)\equiv0$ for all $k\geq1$, and allow a single time-varying source term $p_0(t)\geq0$ at $k=0$. Then
\begin{equation}\label{SourceTermExplicit}
{}[S_0](t)=e^{-\int_0^tp_0(\tau)d\tau},\qquad [S_k](t)=1-\int_0^t\frac{(q\tau)^{k-1}}{(k-1)!}qe^{-q\tau}\left(1-e^{-\int_0^{t-\tau}p_0(\tau')d\tau'}\right)d\tau,\qquad  k=1,2\dots
\end{equation}
In particular,
\begin{equation}\label{SourceTermExplicit_limit}
\lim\limits_{t\to+\infty}[S_k](t)=e^{-\int_0^{+\infty}p_0(\tau)d\tau},\qquad k=0,1,\dots
\end{equation}
\end{proposition}

\begin{proof}
Under the above assumptions, equation~\eqref{EquationDiscrete1D1SidedWithP_Heterogeneous} reduces to
\begin{equation}\label{aaa}
{}[S_0]'(t)+p_0(t) [S_0](t)=0,\qquad S_0(0)=1,
\end{equation}
and
\begin{equation}\label{bbb}
{}[S_k]'(t)+q\left([S_k](t)-[S_{k-1}](t)\right)=0,\qquad S_k(0)=1,\qquad  k=1,2\dots
\end{equation}
The expression of $[S_0](t)$ in~\eqref{SourceTermExplicit} is obtain by a straightforward integration of~\eqref{aaa}.
Using~\eqref{S+I=1}, equation~\eqref{bbb} yields
$$
{}[I_k]'(t)+q\left([I_k](t)-[I_{k-1}](t)\right)=0,\qquad I_k(0)=0,\qquad  k=1,2\dots
$$
Hence,
\begin{equation}\label{Recursion1}
{}[I_k](t)=q e^{-q\cdot}\star_t [I_{k-1}](\cdot),\qquad \text{ }e^{-q\cdot}\star_t [I_{k-1}](\cdot):= \int_0^t e^{-q(t-\tau)} [I_{k-1}](\tau)d\tau.
\end{equation}
By induction, we infer that
\begin{equation}\label{LocalizedSource_IntermediateExpression}
{}[I_{k}](t)= q^k\underbrace{e^{-q\cdot}\star_t\cdots\star_t e^{-q\cdot}}_{k\text{ terms}}\star_t[I_{0}](\cdot),
\end{equation}
where we have used the associativity of the convolution product.
Now,
\begin{equation}\label{LocalizedSource_Step1}
{}[I_0](t)=1-[S_0](t)=1-e^{\int_0^t p_0(\tau)d\tau}.
\end{equation}
In addition, a straightforward induction shows that 
\begin{equation}\label{LocalizedSource_Step2}
q^k\underbrace{e^{-q\cdot}\star_t\cdots\star_t e^{-q\cdot}}_{k\text{ terms}}= \frac{(qt)^{k-1}}{(k-1)!}qe^{-qt}.
\end{equation}
Plugging~\eqref{LocalizedSource_Step1} and~\eqref{LocalizedSource_Step2} into~\eqref{LocalizedSource_IntermediateExpression} gives~\eqref{SourceTermExplicit}.

Let us now prove~\eqref{SourceTermExplicit_limit}.
For a fixed $k\in\{0,1,\dots\}$, by the dominated convergence theorem, we have
$$
\lim\limits_{t\to+\infty}\int_0^t\frac{(q\tau)^{k-1}}{(k-1)!}qe^{-q\tau}\left(1-e^{-\int_0^{t-\tau}p_0}\right)d\tau
=
\left(1-e^{-\int_0^{+\infty}p_0}\right)\int_0^{+\infty}\frac{(q\tau)^{k-1}}{(k-1)!}qe^{-q\tau}d\tau.
$$
Since
$$\int_0^{+\infty}\frac{(q\tau)^{k-1}}{(k-1)!}qe^{-q\tau}d\tau=1,$$
this concludes the proof of~\eqref{SourceTermExplicit_limit}.
\end{proof}
The limit~\eqref{SourceTermExplicit_limit} does not depend on $k$ and, in particular, coincides with the limit as $t\to+\infty$ of $[S_0](t)=e^{-\int_0^{t}p_0}$. This limit can be obtained formally by passing to the limit in~\eqref{EquationDiscrete1D1SidedWithP_Heterogeneous}.
Note also that, for a fixed $t\geq0$, as $k\to+\infty$ in~\eqref{SourceTermExplicit}, we get that $[S_k](t)\to 1$. This expresses the fact that, with probability $1$,
the source term located at $k=0$ does not influence far-away nodes in finite time, since the epidemics propagates at a finite speed (see \autoref{th:SpeedPropagation1D1S}).

{
Note that by formally putting $p_0(t):=K\delta_{t=0}$ in~\eqref{SourceTermExplicit},  where $K>0$ and $\delta_{t=0}$ denotes the Dirac mass centered at $0$, then we recover the formula~\eqref{Explicit_withP_Uk(t)} with $p=0$ for the patient-zero problem as $K\to+\infty$.
}
%


\section{SIR-Bass model on 1D one-sided lattices}\label{sec:1D1S_SIR}
We now consider the more general 1D one-sided case, when recovery is not neglected, see~\eqref{Assumption_1D1S_WithRecovery_Intro}.

\subsection{Deterministic desciption}
We begin by deriving from the discrete stochastic model~\eqref{Assumption_1D1S_WithRecovery_Intro} a system of differential equations for $[S_k](t)$:
\begin{theorem}\label{th:1D1SWithP_R}
Assume that the nodes are placed on a 1D one-sided lattice, see~\eqref{Assumption1D1S}, that the stochastic dynamics are governed by~\eqref{Assumption_1D1S_WithRecovery_Intro}, and that the initial conditions are stochastic and uncorrelated, see~\eqref{hyp:initial_cond_uncor}.
{
For any $k\in\mathcal{K}$, if $q_k(t)>0$ for all $t\geq0$, then $[S_k](t)$ satisfies 
\begin{subequations}
\begin{equation}\label{EquationDiscrete1D1SidedWithP_R_Dt}
\begin{aligned}
{}
&[S_k]'+\big(p_k+q_k+r_k\big)[S_k]= q_ke^{-\int_0^tp_k}[S_k^0]\left([S_{k-1}]+ [R_{k-1}^0]+\frac{r_k(0)}{q_k(0)}+\int_0^t \left(\frac{r_k}{q_k}\right)'(s)\frac{[S_k](s)}{[S_k^0]}ds\right),
\\
& [S_k](0)=[S_k^0],
\end{aligned}
\end{equation}
and $[R_k]$ satisfies
\begin{equation}\label{EquationDiscrete1D1SidedWithP_R_Dt_R}
{}[R_k]'(t)-r_k\big(1-[S_k]-[R_k]\big)=0,\qquad [R_k](0)=[R_k^0].
\end{equation}
\end{subequations}
}
\end{theorem}

\begin{proof}
In the following calculations, we frequently make use of the identities
\begin{equation}\label{RelationClasssique}
\begin{gathered}
\P(A\cap B)=\P(A\vert B)\P(B),\qquad \P(A_1\cap A_2\vert B)=\P(A_1\vert A_2\cap B)\P(A_2\vert B),
\\
 \P(A_1\cap A_2\vert B)\P(B)=\P(A_1\vert A_2\cap B)\P(A_2\cap B).
 \end{gathered}
\end{equation}
We have that
\begin{align*}
\P\( S_k^n\)-\P\( S_k^{n+1}\)
&=\P\(I_k^{n+1}\cap S_k^n\)\\
&=\P\(I_k^{n+1}\cap S_k^n\cap S_{k-1}^n\)+\P\(I_k^{n+1}\cap S_k^n\cap I_{k-1}^n\)+\P\(I_k^{n+1}\cap S_k^n\cap R_{k-1}^n\)\\
&= \underbrace{\P\left({I_k^{n+1}}\big\vert S_k^n\cap S_{k-1}^n \right)}_{=p_k^n\dt}\P\left( S_k^n\cap S_{k-1}^n\)
+ \underbrace{\P\left({I_k^{n+1}}\big\vert S_k^n\cap I_{k-1}^n\right)}_{=(p_k^n+q_k^n)\dt}\P\left(S_k^n\cap I_{k-1}^n\)\\
&\qquad + \underbrace{\P\left({I_k^{n+1}}\big\vert S_k^n\cap R_{k-1}^n\right)}_{=p_k^n\dt}\P\left(S_k^n\cap R_{k-1}^n\).
\end{align*}
and so
\begin{equation}\label{1D1S_R_Proof_1}
\frac{\P\( S_k^{n+1}\)-\P\( S_k^n\)}{\dt}+p_k^n \P(S_k^n)+q_k^n\P\(S_k^n\cap I_{k-1}^n\)=0.
\end{equation}

Using the closure relation~\eqref{1D1S_P_Closure}, we deduce that
\begin{equation}\label{1D1S_P_R_Proof_RelationClosure}
\begin{aligned}
\P\(S_k^n\cap I_{k-1}^n\)
&= \P(S_k^n)- \P\(S_k^n\cap S_{k-1}^n\)-\P\(S_k^n\cap R_{k-1}^n\)\\
&= \P(S_k^n)- e^{-\int_0^{t^n} p_k(s)ds}\P(S_k^0)\P(S_{k-1}^n)(1+O(\dt))-\P\(S_k^n\cap R_{k-1}^n\).
\end{aligned}
\end{equation}
Substituting this expression into~\eqref{1D1S_R_Proof_1}, we deduce
\begin{equation}\label{1D1S_R_P_Proof_EquationOnS_k}
\frac{\P\( S_k^{n+1}\)-\P\( S_k^n\)}{\dt}+(p_k^n+q_k^n)\P(S_k^n) - q_k^ne^{-\int_0^{t^n} p_k(s)ds}\P(S_k^0)\P(S_{k-1}^n) - q_k^n\P\(S_k^n\cap R_{k-1}^n\)=O(\dt).
\end{equation}
{
Letting $\dt\to0$, we finally deduce
\begin{equation}\label{1D1S_R_P_Proof_EquationOnS_k_limite}
\frac{d[S_k]}{dt}(t)+(p_k(t)+q_k(t))[S_k] - q_ke^{-\int_0^{t} p_k}[S_k^0][S_{k-1}] - q_k[S_k\cap R_{k-1}](t)=0.
\end{equation}
}

\paragraph*{}
Next, let us derive an equation for $[S_k\cap R_{k-1}]$:
\begin{align*}
&\P\( S_k^{n+1}\cap R_{k-1}^{n+1}\)-\P\( S_k^{n}\cap R_{k-1}^n\)\\
&= \P\( S_k^{n+1}\cap R_{k-1}^{n+1}\)-\P\( S_k^{n+1}\cap R_{k-1}^n\)+ \P\( S_k^{n+1}\cap R_{k-1}^n\) -\P\( S_k^{n}\cap R_{k-1}^n\)\\
&=  \P\(S_k^{n+1}\cap I_{k-1}^n\cap R_{k-1}^{n+1}\)           -          \P\(I_k^{n+1}\cap S_k^n\cap R_{k-1}^n\)\\
&= \underbrace{\P\left({R_{k-1}^{n+1}}\big\vert S_k^{n+1}\cap I_{k-1}^n\right)}_{=r_k^n\dt}\P\left(S_k^{n+1}\cap I_{k-1}^n\)        
-     \underbrace{\P\left({I_k^{n+1}}\big\vert S_k^n\cap R_{k-1}^n\right)}_{=p_k^n\dt}\P\left(S_k^n\cap R_{k-1}^n\)\\
&=r_k^n\dt\P\left(S_k^{n+1}\cap I_{k-1}^n\)       -    p_k^n\dt\P\left(S_k^n\cap R_{k-1}^n\)\\
&= r_k^n\dt\P\left(S_k^{n}\cap I_{k-1}^n\)       -    p_k^n\dt\P\left(S_k^n\cap R_{k-1}^n\) +r_k^n\dt\left[\P\left(S_k^{n+1}\cap I_{k-1}^n\) -\P\left(S_k^{n}\cap I_{k-1}^n\) \right]\\
&=r_k^n\dt\left[\P\left(S_k^{n}\)-\P\left(S_k^{n}\cap S_{k-1}^n\)-\P\left(S_k^{n}\cap R_{k-1}^n\)\right] -    p_k^n\dt\P\left(S_k^n\cap R_{k-1}^n\)-r_k^n\dt\P\left({I_k^{n+1}}\cap S_k^{n}\cap I_{k-1}^n\) 
\end{align*}
Using~\eqref{1D1S_P_Closure} and the relation
\begin{align*}
\P\left(\overline{S_k^{n+1}}\cap S_k^{n}\cap I_{k-1}^n\) 
=\P\left(I_k^{n+1}\cap S_k^{n}\cap I_{k-1}^n\)
&=\P\left(I_k^{n+1}\vert S_k^{n}\cap I_{k-1}^n\)\P\left(S_k^{n}\cap I_{k-1}^n\)\\
&= (p_k^n+q_k^n)\dt \P\left(S_k^{n}\cap I_{k-1}^n\),
\end{align*}
we deduce that
\begin{equation}\label{1D1S_R_P_Proof_EquationOnS_kR_k1_Dt}
\begin{aligned}
&\frac{\P\( S_k^{n+1}\cap R_{k-1}^{n+1}\)-\P\( S_k^{n}\cap R_{k-1}^n\)}{\dt}-r_k^n\P\left(S_k^{n}\) +r_k^n\(1-p_k^n\dt\)^n\P(S_k^0)\P\( S_{k-1}^n\) + (p_k^n+r_k^n)\P\left(S_k^n\cap R_{k-1}^n\)\\
&\qquad \qquad =-r_k^n(p_k^n+q_k^n)\dt \P\left(S_k^{n}\cap I_{k-1}^n\).
\end{aligned}
\end{equation}
{
Letting $\dt\to0$, we deduce
\begin{equation}\label{1D1S_R_P_Proof_EquationOnS_kR_k1_Dt_limite}
\frac{d[S_k\cap R_{k-1}]}{dt}(t)-r_k(t)[S_k](t) +r_ke^{-\int_0^t p(s)ds}[S_k^0][S_{k-1}] + (p_k+r_k)[S_k\cap R_{k-1}]=0.
\end{equation}
}

\paragraph*{}
{
At this stage, equations~\eqref{1D1S_R_P_Proof_EquationOnS_k_limite} and~\eqref{1D1S_R_P_Proof_EquationOnS_kR_k1_Dt_limite} form an autonomous system of ODEs governing the evolution of the family of probabilities $\{[S_k],[S_k\cap R_{k-1}]\}_{k\in\mathcal{K}}$. This system can actually be reduced to an autonomous system involving only the probabilities $\{[S_k]\}_{k\in\mathcal{K}}$.
Let us set
$$w_k(t):=\left(\frac{r_k(t)}{q_k(t)}[S_k](t)+[S_k\cap R_{k-1}](t)\right)e^{\int_0^t p_k}.$$
Using equations~\eqref{1D1S_R_P_Proof_EquationOnS_k_limite} and\eqref{1D1S_R_P_Proof_EquationOnS_kR_k1_Dt_limite} satisfied by $[S_k]$ and $[S_k\cap R_{k-1}]$, we find that $w_k$ satisfies
 satisfies
\begin{equation*}
w_k'(t)=\left(\frac{r_k}{q_k}\right)'(t)[S_k](t).
\end{equation*}
Therefore,
\begin{equation*}
w_k(t)=w_k^0+\int_0^t \left(\frac{r_k}{q_k}\right)'(s)[S_k](s)ds,
\end{equation*}
and so
\begin{equation}\label{Relation_SkRk1_Sk}
{}[S_k\cap R_{k-1}]=-\frac{r_k}{q_k}[S_k]+ e^{-\int_0^t p_k}\left(\frac{r_k(0)}{q_k(0)}[S_k^0]+[S_k^0][R_{k-1}^0]+\int_0^t \left(\frac{r_k}{q_k}\right)'(s)[S_k](s)ds \right).
\end{equation}
Substituting the above expression in~\eqref{1D1S_R_P_Proof_EquationOnS_k},
we find~\eqref{EquationDiscrete1D1SidedWithP_R_Dt}.
}

To derive equation~\eqref{EquationDiscrete1D1SidedWithP_R_Dt_R} for $[R_k](t)$, we note that
\begin{align*}
\P\(R_k^{n+1}\)-\P\(R_k^n\)
= \P\(R_k^{n+1}\cap I_k^{n}\)
=\P\(R_k^{n+1}\vert I_k^{n}\)\P\( I_k^{n}\)
= r_k^n\dt \P\( I_k^{n}\).
\end{align*}
Therefore,
\begin{equation*}
\frac{\P\(R_k^{n+1}\)-\P\(R_k^n\)}{\dt}-r_k^n\P\( I_k^{n}\)=0.
\end{equation*}
Using~\eqref{S+I+R=1} and letting $\dt\to0$ in the above equality gives~\eqref{EquationDiscrete1D1SidedWithP_R_Dt_R}.
\end{proof}

\begin{remark}
{
Equations~\eqref{EquationDiscrete1D1SidedWithP_R_Dt} form a system of integrodifferential equation that only involves the probabilities $\{[S_k](t)\}_{t\geq0,k\in\mathcal{K}}\}$. If the ratio $\frac{r_k}{q_k}$ does not depend on $t$, it boils down to the system of ODEs
\begin{equation}\label{EquationDiscrete1D1SidedWithP_R_Dt_NOTIME}
{}
[S_k]'+\big(p_k+q_k+r_k\big)[S_k]= e^{-\int_0^tp_k}[S_k^0]\left(q_k[S_{k-1}]+ q_k[R_{k-1}^0]+r_k \right),
\qquad
 [S_k](0)=[S_k^0].
\end{equation}
}
An apparently different system of ODEs for $[S_k]$ and $[R_k]$ was derived in~\cite{Fibich2016} for the case of the deterministic initial conditions $[S_k^0]\equiv1$. 
In that system, however, the equation for $[S_k]$ was not decoupled from that for $[R_k]$.
\end{remark}

\autoref{th:1D1SWithP_R} gives a complete deterministic description of the stochastic process, since once $[S_k]$ and $[R_k]$ are computed through equations~\eqref{EquationDiscrete1D1SidedWithP_R_Dt}-\eqref{EquationDiscrete1D1SidedWithP_R_Dt_R}, $[I_k]$ is deduced from~\eqref{S+I+R=1}.

\subsection{Explicit solutions}

We can solve equations~\eqref{EquationDiscrete1D1SidedWithP_R_Dt}-\eqref{EquationDiscrete1D1SidedWithP_R_Dt_R} explicitely in the homogeneous case.
\begin{proposition}[Spatially homogeneous solution]\label{th:Coro_Homogeneous_1D1S_General}
Assume the conditions of \autoref{th:1D1SWithP_R}, and assume further that the initial conditions $[S_k^0]\equiv [S^0]$ do not depend on the space variable $k$ {and that the parameters $p_k(t)\equiv p$, $q_k(t)=q>0$ and $r_k(t)$ do not depend on the time variable $t$ nor the space variable $k$. Then the solutions $[S_k]=[S]$ and $[R_k]=[R]$ of~\eqref{EquationDiscrete1D1SidedWithP_R_Dt}-\eqref{EquationDiscrete1D1SidedWithP_R_Dt_R} do not depend on $k$, and are given by
\begin{equation}\label{Expression_S_Uniform_SIR}
{}
[S](t)=[S^0]e^{-\left(p+q+r\right)t+[S^0]q\frac{1-e^{-pt}}{p}}\left(1+\left(r+q[R^0]\right)\int_0^te^{(q+r)\tau-q[S^0]\frac{1-e^{-p\tau}}{p}} d\tau\right),
\end{equation}
\begin{equation}\label{Expression_R_Uniform_SIR}
{}
[R](t)=1-(1-[R^0])e^{-rt}-re^{-rt}
\int_0^te^{r\tau}[S](\tau)d\tau.
\end{equation}
}

\end{proposition}
\begin{proof}
By translation symmetry, $[S_k]$ and $[R_k]$ do not depend on $k$. The result follows from a straighforward integration of equations~\eqref{EquationDiscrete1D1SidedWithP_R_Dt}-\eqref{EquationDiscrete1D1SidedWithP_R_Dt_R}.
\end{proof}

\begin{remark}
Expression~\eqref{Expression_S_Uniform_SIR} was already obtained in~\cite[Lemma~4.2]{Fibich2017} for $[S^0]\equiv1$ and $[R^0]\equiv0$.
\end{remark}

We can also explicitely solve the case of a time-dependent source term located at $k=0$.
\begin{proposition}[SIR-Bass time-varying point source]\label{th:SourceTerm_R}
Assume the conditions of~\autoref{th:1D1SWithP_R}, that the individuals are placed on a semi infinite line $\mathcal{K}=\{0,1,2,\dots\}$, that $q_k(t)\equiv q>0$ is constant, that $p_k(t)\equiv0$ for all $k>0$, and allow a source term $p_0(t)\geq0$ at $k=0$. Then
\begin{equation}\label{SourceTermExplicit_R}
{}[S_0](t)=e^{-\int_0^tp_0},\qquad [S_k](t)=1-\int_0^t\frac{(q\tau)^{k-1}}{(k-1)!}qe^{-(q+r)\tau}\left(1-e^{-\int_0^{t-\tau}p_0}\right)d\tau,\qquad k=1,2\dots
\end{equation}
In particular,
{
\begin{equation}\label{SourceTermExplicit_limit_R}
\lim\limits_{t\to+\infty}[S_k](t)=1-\left(\frac{q}{q+r}\right)^k\left(1-e^{-\int_0^{+\infty}p_0}\right), \qquad k=0,1,2,\dots
\end{equation}
}
\end{proposition}

\begin{proof}
Under our assumptions, equation~\eqref{EquationDiscrete1D1SidedWithP_R_Dt} reduces to
$$
{}[S_0]'(t)+p_0(t)[S_0](t)=0,\qquad [S_0](0)=1,
$$
and
$$
{}[S_k]'(t)+q\left([S_k](t)-[S_{k-1}](t)\right)+r\left([S_k](t)-1\right)=0,\qquad [S_k](0)=0,\qquad k=1,2\dots
$$
Therefore, $[S_0](t)=e^{-\int_0^tp_0}$ and it can be verified by direct substitution that the solution for $k\geq1$ is given by
\begin{equation}\label{Recursion1_R}
{}[S_k](t)=e^{-(q+r)t}+ e^{-(q+r)\cdot}\star_t \left(r+q[S_{k-1}](\cdot)\right),
\end{equation}
where
$$
f(\cdot)\star_t g(\cdot):= \int_0^t f(t-\tau) g(\tau)d\tau.
$$
Note also that the constant $1$ is a fixed point of the recursion relationship~\eqref{Recursion1_R}, namely,
\begin{equation*}
{}1=e^{-(q+r)t}+ e^{-(q+r)\cdot}\star_t \left(r+q\right),
\end{equation*}
Therefore, for all $k>0$,
\begin{equation*}
{}[S_{k}](t)-1 = q e^{-(q+r)\cdot}\star_t\left([S_{k-1}](\cdot)-1\right).
\end{equation*}
By induction, we infer
\begin{equation*}
{}[S_{k}](t)=1+ q^k\underbrace{e^{-(q+r)\cdot}\star_t\cdots\star_t e^{-(q+r)\cdot}}_{k\text{ terms}}\star_t\left([S_{0}](\cdot)-1\right),
\end{equation*}
where we have used the associativity of the convolution product.
A straightforward induction shows that $$
q^k\underbrace{e^{-(q+r)\cdot}\star_t\cdots\star_t e^{-(q+r)\cdot}}_{k\text{ terms}}= \frac{(qt)^{k-1}}{(k-1)!}qe^{-(q+r)t}.
$$
Therefore
\begin{equation*}
{}[S_{k}](t)=1+ \left(\tau\mapsto\frac{(q\tau)^{k-1}}{(k-1)!}qe^{-(q+r)\tau}\right)\star_t\left([S_{0}](\cdot)-1\right).
\end{equation*}
Using 
$
[S_0](t)=e^{-\int_0^tp_0},
$
we deduce that
\begin{equation*}
{}[S_{k}](t)=1+ \left(\tau\mapsto \frac{(q\tau)^{k-1}}{(k-1)!}qe^{-(q+r)\tau}\right)\star_t\left(\tau\mapsto e^{-\int_0^\tau p_0}-1\right).
\end{equation*}
which proves~\eqref{SourceTermExplicit_R}.

Let us now prove~\eqref{SourceTermExplicit_limit_R}.
For a fixed $k\in\{0,1,\dots\}$, by the dominated convergence theorem,
$$
\lim\limits_{t\to+\infty}\int_0^t\frac{(q\tau)^{k-1}}{(k-1)!}qe^{-(q+r)\tau}\left(1-e^{-\int_0^{t-\tau}p_0}\right)d\tau
=
\left(1-e^{-\int_0^{+\infty}p_0}\right)\int_0^{+\infty}\frac{(q\tau)^{k-1}}{(k-1)!}qe^{-(q+r)\tau}d\tau.
$$
Using the relation
$$\int_0^{+\infty}\frac{(q\tau)^{k-1}}{(k-1)!}qe^{-(q+r)\tau}d\tau=\left(\frac{q}{q+r}\right)^k$$
completes the proof of~\eqref{SourceTermExplicit_limit_R}.
\end{proof}

Thus, the limit~\eqref{SourceTermExplicit_limit_R} depends on space $k$ if and only if $r>0$.
Note also that, for a fixed $t\geq0$, as $k\to+\infty$ in~\eqref{SourceTermExplicit_R}, we get that $[S_k](t)\to 1$. This expresses again the fact that a source term located at $k=0$ cannot influence far nodes in finite time. 

{
From~\eqref{SourceTermExplicit_limit_R}, we observe that
\begin{equation}
\frac{1-\lim\limits_{t\to+\infty}[S_{k+1}](t)}{1-\lim\limits_{t\to+\infty}[S_{k}](t)}=\frac{q}{q+r}.
\end{equation}
This expression is actually equal to the probability that node $k+1$ will eventually become infected knowing that node $k$ has been infected.
}

\subsection{$SIR$ model on 1D one-sided lattices}\label{sec:Comparison_SIR}
When $p_k^n\equiv 0$, i.e., when there are no source terms, the model~\eqref{Assumption_1D1S_WithRecovery_Intro} reduces to the $SIR$ model on a 1D one-sided lattice. Assuming that the parameters and initial conditions are homogeneous (they do not depend on time nor space), if we let $p\to0$ in~\eqref{Expression_S_Uniform_SIR}-\eqref{Expression_R_Uniform_SIR}, then $\frac{1-e^{-pt}}{p}\to t$, and so we obtain
\begin{equation*}
{}[S](t) =\frac{[S^0]}{r+q(1-[S^0])}\left( r+q[R^0]+q[I^0]e^{-\left(r+q(1-[S^0])\right)t}\right),
\end{equation*}
and
{
\begin{equation*}
{}
[R](t)=1-[S^0]\frac{r+q[R^0]}{r+q(1-[S^0])}-\frac{[I^0]}{1-[S^0]}e^{-rt}+\frac{r[S^0][I^0]}{r+q(1-[S^0])(1-[S^0])}e^{-(r+q(1-[S^0]))t}.
\end{equation*}
}
Therefore,
\begin{equation}\label{limit_Sinfty}
{}[S^\infty]:=\lim\limits_{t\to+\infty}[S](t)= [S^0]\frac{r+q[R^0]}{r+q(1-[S^0])},
\end{equation}
and
$${}[R^\infty]:=\lim\limits_{t\to+\infty}[R](t)= 1-[S^0]\frac{r+q[R^0]}{r+q(1-[S^0])}=1-[S^\infty],
$$
and so, by~\eqref{S+I+R=1}, we deduce that $[I^\infty]:=\lim\limits_{t\to+\infty}[I](t)=0.$

\paragraph*{}
The SIR model on a 1D one-sided lattice~\eqref{Assumption_1D1S_WithRecovery_Intro} satisfies a threshold phenomenon for the initial outbreak of an epidemic, namely, an initially small number of infected ($S^0\approx 1$, $0<I^0\ll1$, $R^0=0$) triggers an epidemic if and only if
\begin{equation}\label{threshold_ClassicalSIR}
q-r>0.
\end{equation}
Indeed, under the same conditions as in \autoref{th:Coro_Homogeneous_1D1S_General}, if we further assume that $p=0$ and $[R^0]=0$, using~\eqref{S+I+R=1}, \eqref{EquationDiscrete1D1SidedWithP_R_Dt} and \eqref{EquationDiscrete1D1SidedWithP_R_Dt_R}, we find that $[I]$ satisfies 
\begin{equation*}
{}[I]'-\big(r+q(1-[S^0])\big)(1-[I]) +r[S^0]+r[I]=0,
\end{equation*}
i.e.,
\begin{equation*}
{}[I]'-(r+q)[I^0]+\big(2r+q[I^0]\big)[I]=0.
\end{equation*}
If we linearize this equation around $0<[I]\approx [I^0]\ll1$, we find that $[I]$ increases in small time if and only if~\eqref{threshold_ClassicalSIR} holds.

\paragraph*{}
Let us now solve explicitely the ``patient-zero'' problem for the $SIR$ model:
\begin{proposition}[Patient-zero $SIR$ problem]\label{th:coro_patient_zero_SIR}
Assume that $p_k^n\equiv0$, $q_k^n\equiv q$ and $r_k^n\equiv r$ in~\eqref{Assumption_1D1S_WithRecovery_Intro}, that the nodes are placed on a semi-inifinite line $\mathcal{K}=\{0,1,2,\dots\}$ and that, initially,
\begin{equation}\label{Initial_Condition_PatientZero_1D1S_General}
x_{k=0}^0=i,\qquad x_k^0=s,\quad k=1,2,\dots
\end{equation}
Then, the solution of~\eqref{EquationDiscrete1D1SidedWithP_R_Dt} is given by
\begin{equation}\label{Explicit_SIR_Sk(t)}
{}[S_0](t)\equiv0,\qquad [S_k](t)=1-\left(\frac{q}{q+r}\right)^k\left(1- e^{-(q+r)t}\sum\limits_{l=0}^{k-1} \frac{\left((q+r) t\right)^l}{l!}\right),\qquad k=1,2,\dots
\end{equation}

If we denote by $N_K(t)=\sum_{k=1}^K \mathds{1}_{x_k(t)=s}$ the total number of susceptibles on a finite line of length $K+1>0$, then
\begin{equation}\label{Explicit_SIR_Ek}
\mathbb{E}[N_K^n]=K- \frac{q}{r}\left(1-\left(\frac{q}{q+r}\right)^K\right)- \frac{q+r}{r} e^{-(q+r)t} \sum\limits_{l=1}^{K}\frac{\left(q t\right)^{K-l}}{(K-l)!} \left(1-\left(\frac{q}{q+r}\right)^{l+1}\right).
\end{equation}
\end{proposition}
\begin{proof}
See Appendix~\ref{sec:Proof_of_th:coro_patient_zero_SIR}.
\end{proof}

We see from~\eqref{Explicit_SIR_Sk(t)} that $\lim_{k\to+\infty}[S_k(t)]=1$, i.e., that the effect of patient-zero vanishes as we move far away from $k=0$.

\paragraph*{}
 It is instructive to compare the $SIR$ model on a 1D-one sided lattice to the classical aggregate $SIR$ model~\eqref{ClassicalSIR}.
Since~\eqref{ClassicalSIR} corresponds to a complete graph whereas \eqref{Assumption_1D1S_WithRecovery_Intro} corresponds to a structured sparse graph, this comparison reveals the effect of the network structure on the dynamics.
\begin{itemize}
\item First, we emphasize that equation~\eqref{EquationDiscrete1D1SidedWithP_R_Dt} for $[S_k]$ does not include $[I_k]$ and $[R_k]$, or higher-order marginals such as $[S_k\cap S_{k-1}]$. 
This property is unique to 1D one-sided graphs. Indeed, for the aggregate $SIR$ model~\eqref{ClassicalSIR}, it is not possible to derive an equation on $S$ which does not involve $I$ and $R$.

Let us mention that the system~\eqref{SIR-Bass}, and therefore also~\eqref{ClassicalSIR} obtained from~\eqref{SIR-Bass} by taking $p=0$, can also be reduced to a single equation on $u(t):=\int_0^t I$. Indeed, integrating the equation for $S$ in~\eqref{SIR-Bass}, we obtain
$
S(t)=S^0 e^{-pt-q\int_0^t I}.
$
Injecting this expression into the equation for $I$ in~\eqref{SIR-Bass}, integrating the equation on $(0,t)$, we obtain
$$
u'(t)=I^0+S^0\left(1-e^{-pt-qu(t)}\right)-ru(t).
$$

\item  The aggregate $SIR$ model satisfies the same threshold condition~\eqref{threshold_ClassicalSIR} for the outbreak of an epidemic from a small initial number of infected, see~\cite{Hethcote1989}.
\item When \eqref{threshold_ClassicalSIR} holds, the behavior after the outbreak of an epidemic can be quite different for the 1D one-sided case~\eqref{Assumption_1D1S_WithRecovery_Intro} and the aggregate model~\eqref{ClassicalSIR}. In particular, a remarkable difference is the dependence of the final state on the initial conditions.
From~\eqref{limit_Sinfty}, we observe that 
\begin{enumerate}
\item $[S^\infty]$ is an increasing function of $[S^0]$: the more susceptible individuals at initial time, the more for large times,
\item If $0<[I^0],[R^0]\ll1$ then $[S^\infty]\approx [S^0]\approx 1.$
\end{enumerate}

In contrast, for the aggregate $SIR$ model~\eqref{ClassicalSIR}, we have that
\begin{enumerate}
\item $[S^\infty]$ is a \emph{nonincreasing} function $[S^0]$, see~\cite{Hethcote1989},
\item $[I^0], [R^0]\ll1$ does not imply that $[S^\infty]\approx [S^0]\approx 1$.
\end{enumerate}
\end{itemize}
%
%
%
%

\section{SIR-Bass model on 1D two-sided lattices}\label{sec:1D2S}

We now allow both the left and the right neighbors to spread the infection. We denote by$q_k^{L,n}:=q_{k-1,k}$ and $q_k^{R,n}:=q_{k+1,k}$ the influence of the left and right neighbors on node $k$, respectively, and we assume that
\begin{equation}\label{Structure_Assumption_1D2S}
q_{ik}=0\qquad \text{if }\vert i-k\vert\neq1.
\end{equation}
 Hence,~\ref{Assumption_Generale_1_Intro_BIS} becomes
\begin{subequations}\label{Assumption_1D2S_WithRecovery}
\begin{equation}\label{Assumption_1D2S_WithRecovery_1}
 \tag*{(\ref{Assumption_1D2S_WithRecovery}a)}
\P\left(I_k^{n+1}\big\vert X^n \right)=
\left\{\begin{aligned}
&\dt\left(p_k^n+q_k^{L,n} \mathds{1}_{I_{k-1}^n}+q_k^{R,n} \mathds{1}_{I_{k+1}^n}\right),&&\text{if }x_k^n=s,\\
& 1-r_k^n\dt,&&\text{if }x_k^n=i,\\
& 0 ,&&\text{if }x_k^n=r,
\end{aligned}\right.
\end{equation}
The transition rates~\ref{Assumption_Generale_2_Intro_BIS} from $i$ to $r$ remain unchanged, that is:
\begin{equation}\label{Assumption_1D2S_WithRecovery_2}
 \tag*{(\ref{Assumption_1D2S_WithRecovery}b)}
\P\left(R_k^{n+1}\big\vert X^n \right)=
\left\{\begin{aligned}
&0,&&\text{if }x_k^n=s,\\
& r_k^n\dt,&&\text{if }x_k^n=i,\\
& 1 ,&&\text{if }x_k^n=r,
\end{aligned}\right.
\end{equation}
\end{subequations}

\subsection{Deterministic description}

\begin{theorem}\label{th:1D2SWithP_R}
Assume that the individuals are placed on a 1D two-sided lattice $\mathcal{K}$ {(which is either a finite, semi-infinite, or infinite line)}, see~\eqref{Structure_Assumption_1D2S}, that the dynamics are governed by~\eqref{Assumption_1D2S_WithRecovery}, and that the initial conditions are uncorrelated, see~\eqref{hyp:initial_cond_uncor}. Let $p_k^n$, $q_k^n$ and $r_k^n$ converge as $\dt\to0$ to $p_k(t)$, $q_k(t)$ and $r_k(t)$.
Then for all $k\in\mathcal{K}$,
\begin{subequations}
\begin{equation}\label{EquationDiscrete1D2Sided_1}
{}[S_k](t)=
\left\{\begin{aligned}
&\frac{[S_k^L](t)[S_k^R](t)}{[S_k^0]e^{-\int_0^t p_k(\cdot)}}, &&\text{if }[S_k^0]>0,\\
&0, &&\text{if }[S_k^0]=0,
\end{aligned}\right.
\end{equation}
where $[S_k^R](t)$ (resp. $[S_k^L](t)$) is the probability that node $k$ is susceptible by time $t$ if we discard the influences of the left neighbors by setting $q^{L,n}_k\equiv0$ in~\eqref{Assumption_1D2S_WithRecovery_1} (resp. the influences of the right neighbors by setting $q^{R,n}_k\equiv0$).
In addition, $[R_k(t)]$ is given by
\begin{equation}\label{EquationDiscrete1D2Sided_3}
{}[R_k]'(t)-r_k(t)\big(1-[S_k]-[R_k]\big)=0.
\end{equation}
\end{subequations}
\end{theorem}

\begin{proof}
Let us first introduce some notations. We denote by $[S_k\vert q_k^L=0]$ the probability of the event $\tilde S_k(t):=\{\tilde x_k(t)=s\}$ where $\tilde x_k$ is the process defined identically as $x_k$ but in which we put $q_k^L\equiv0$ in~\eqref{Assumption_1D2S_WithRecovery_1}. We use the name notation if $S_k$ is replaced by any event or intersection of events, and if $q_k^{L,n}=0$ is replaced by other conditions on the parameters. For example, with these notations, we have 
 $\left[S_k^R\right]= \left[S_k\vert (q_k^L)_{k\in\mathcal{K}}\equiv0\right]$.

First, notice that
\begin{align*}
\left[S_{k-1}\cap S_k\right]
=\left[S_{k-1}\cap S_k\vert q_{k-1}^R=q_k^L=0 \right]
&=\left[S_{k-1}\vert q_{k-1}^R=q_k^L=0 \right]  \left[S_k\vert q_{k-1}^R=q_k^L=0 \right]\\
&=\left[S_{k-1}\vert q_{k-1}^R=0 \right]  \left[S_k\vert q_k^L=0 \right]\\
&=\left[S_{k-1}\vert (q_k^R)_{k\in\mathcal{K}}\equiv 0 \right]  \left[S_k\vert (q_k^L)_{k\in\mathcal{K}}\equiv 0\right].
\end{align*}
We deduce that
\begin{equation}\label{Proof_2S_Identity1}
{}\left[S_{k-1}\cap S_k\right]=\left[S_{k-1}^L\right]\left[S_k^R\right].
\end{equation}
Similarily, we have
\begin{align*}
{}[S_{k-1}\cap S_k\cap S_{k+1}]
&=\left[S_{k-1}\cap S_k\cap S_{k+1}\vert q_{k-1}^R=q_k^L=q_k^R=q_{k+1}^L=0\right ]\\
&=\left[S_{k-1}\vert q_{k-1}^R=0\right ]
\left[S_k\vert q_k^L=q_k^R=0\right ]
\left[S_{k+1}\vert q_{k+1}^L=0\right ]\\ 
&= \left[S_{k-1}\vert (q_k^R)_{k\in\mathcal{K}}\equiv 0\right ]
[S_k^0]e^{-\int_0^t p}
\left[S_{k+1}\vert (q_k^L)_{k\in\mathcal{K}}\equiv 0\right ],
\end{align*}
and so
\begin{equation}\label{Proof_2S_Identity2}
{}[S_{k-1}\cap S_k\cap S_{k+1}] =\left[S_{k-1}^L\right ]
\left[S_{k+1}^R\right ][S_k^0]e^{-\int_0^t p}.
\end{equation}
Relations \eqref{Proof_2S_Identity1} and \eqref{Proof_2S_Identity2} can also be proved using the \emph{indifference principle} of~\cite{Fibich2019}.

The last ingredient in our proof is a \emph{spatial Markovian property} establishing that the state of nodes $k-1$ and $k+1$ are independent if conditioned with respect to $S_k$. This is where the assumption on the graph (namely that it is 1D with no cylces) comes into play.
\begin{lemma}\label{rmk:Markov_1D2S_Proof}
\begin{equation}
x_{k-1}^n\independent x_{k+1}^n\vert S_k^n.
\end{equation}
\end{lemma}
This standard result is proved separately in a slightly more general form, see \autoref{rmk:Markov_1D2S} in appendix~\ref{sec:Appendix_SpatialMarkov}. \autoref{rmk:Markov_1D2S_Proof} can also be deduced from the funnel node theorem~\cite{Fibich2021}.

Using \autoref{rmk:Markov_1D2S_Proof}, we have that
\begin{equation}
{}[S_{k-1}\cap S_k\cap S_{k+1}]= [S_{k-1}\cap S_{k+1}\vert  S_k][S_k]=[S_{k-1}\vert  S_k][S_{k+1}\vert  S_k][S_k]=\frac{[S_{k-1}\cap   S_k][ S_k\cap S_{k+1}]}{[S_k]}.
\end{equation}
Substituting~\eqref{Proof_2S_Identity1} and~\eqref{Proof_2S_Identity2} in the right and left hand-side of the above equation yields equation~\eqref{EquationDiscrete1D2Sided_1} on $[S_k]$.

Equation~\eqref{EquationDiscrete1D2Sided_3} on $[R_k]$ is obtained similarily to equation~\eqref{EquationDiscrete1D1SidedWithP_R_Dt_R} in \autoref{th:1D1SWithP_R}.
\end{proof}

\begin{remark}
In the case without recovery, $([S_k])_{k\in\mathcal{K}}$ and $([S_k\cap S_{k-1}])_{k\in\mathcal{K}}$ satisfy a closed system of ODEs. More precisely, under the same conditions as in~\autoref{th:1D2SWithP_R} and if we further assume that there is no recovery~\eqref{1D1S_Assumption_NoRecovery}, then
\begin{subequations}\label{ClosedSystem1D2S}
\begin{equation}
[S_k]'(t)
+\big(p_k+q_k^R+q_k^L\big)[S_k] 
 -e^{-\int_0^tp_k}[S_k^0]\big(q_k^R[S_k\cap S_{k+1}]
+q_k^L[S_k\cap S_{k-1}]\big)
=0,
\end{equation}
and 
\begin{equation}
\begin{aligned}
[S_k\cap S_{k-1}]'(t)
&= -(p_k+p_{k-1}+q_{k-1}^L+q_k^R)[S_k\cap S_{k-1}]\\
&\qquad+q_k^R\frac{[S_{k-1}\cap S_k][S_k\cap S_{k+1}]}{[S_k]}+q_{k-1}^L \frac{[S_{k-2}\cap S_{k-1}][S_{k-1}\cap S_k]}{[S_{k-1}]}.
\end{aligned}
\end{equation}
\end{subequations}
To prove this, first recall that from \autoref{th:1D1SWithP_R} we have
\begin{gather}
{}[S_k^R]'+\big(p_k+q_k^R+r_k\big)[S_k^R] - e^{-\int_0^tp_k}[S_k^0]\big(q_k^R[S_{k+1}^R]+q_k^R[R_{k+1}^0]+r_k\big)=0,\\
{}[S_k^L]'+\big(p_k+q_k^L+r_k\big)[S_k^L] - e^{-\int_0^tp_k}[S_k^0]\big(q_k^L[S_{k-1}^L]+q_k^L[R_{k-1}^0]+r_k\big)=0\label{EquationSR1}
\end{gather}
Then, we derive~\eqref{ClosedSystem1D2S} from the identities~\eqref{EquationDiscrete1D2Sided_1} and~\eqref{Proof_2S_Identity1}. Note also that
\begin{equation}
\frac{[S_{k-1}\cap S_k][S_k\cap S_{k+1}]}{[S_k]}= e^{-\int_0^t p_k}[S_k^0][S_{k-1}^L][S_{k+1}^R].
\end{equation}
\end{remark}

\subsection{Explicit solutions}

Using \autoref{th:1D2SWithP_R}, the results of the previous sections dealing with the spatially homogeneous case, the patient zero problem, and time-varying point sources can be adapted to the case of a two-sided 1D lattice. We now briefly present the obtained explicit formula. For simplicity, we focus on the Bass model, i.e., we assume that there is no recovery.

%
%

\begin{corollary}[Spatially-homogeneous two-sided Bass solution]\label{rmk:1D2S_NonSpatialSolution}
Assume the conditions of \autoref{th:1D2SWithP_R}. If the initial condition $[S_k^0]\equiv [S^0]$ and the parameters $p_k(t)\equiv p(t),q^R_k(t)\equiv q^R(t), q^L_k(t)\equiv q^L(t)$ do not depend on space $k$, then $[S_k](t)=[S](t)$ does not depend on $k$, and is given by
\begin{equation}\label{ExplicitU_Hetero_T_2S}
{}[S](t)=[S^0]e^{-\int_0^t \left(p(s)+q^R(s)+q^L(s)\right)ds+[S^0]\int_0^t(q^R(s)+q^L(s))e^{-\int_0^s p}ds }.
\end{equation}
\end{corollary}
\begin{proof}
The result is proved by a direct application of \autoref{th:1D2SWithP_R} and~\eqref{ExplicitU_Hetero_T}.
\end{proof}
Note that expression \eqref{ExplicitU_Hetero_T_2S} coincides with~\eqref{ExplicitU_Hetero_T} in the one-sided case by setting $q(t):=q^R(t)+q^L(t)$.

We now turn to the patient-zero problem.
\begin{corollary}[Two-sided patient-zero Bass problem]\label{cor:2D_patient-zero}
Assume the conditions of \autoref{th:1D2SWithP_R}, let $r_k(t)\equiv0$, $p_k(t)\equiv p$, $q_k^R(t)\equiv q^R$, $q_k^L(t)\equiv q^L$, be independent of $k$ and $t$, and let the nodes be placed on a semi-inifinite line $\mathcal{K}=\mathbb{N}$. Assume that, initially, we place a patient zero at $k=0$, i.e.,
\begin{equation*}
[I_0^0]=1,\qquad [S_k^0]=1,\quad k=1,2,\dots
 \end{equation*} 
 Then, $[S_k]$ is given by
 \begin{equation*}
 {}[S_k](t)=e^{-(p+q^R+q^L)t+q^R\frac{1-e^{pt}}{p}}
\sum\limits_{l=0}^{k-1} \frac{\left(q^L \frac{1-e^{-pt}}{p}\right)^l}{l!},\qquad \forall k=1,2,\dots
 \end{equation*}
\end{corollary}
\begin{proof}
Note that $[S_k^R]$ is given by the solution~\eqref{ExplicitU_Hetero_T} of the one-sided Bass homogeneous problem with $q=q^R$, and that $[S_k^L]$ is given by the solution~\eqref{Explicit_withP_Uk(t)} of the one-sided Bass patient-zero problem with $q=q^L$. We conclude by using~\eqref{EquationDiscrete1D2Sided_1}.
\end{proof}
%
%

We now present the solution in the case of a time-varying point source.
\begin{corollary}[Time-varying point source]\label{th:SourceTerm_2S}
Assume the conditions of \autoref{th:1D2SWithP_R}, that the individuals are placed on a semi infinite line $\mathcal{K}=\{0,1,2,\dots\}$ and that all individuals are initially susceptible, i.e., $x_k^0=s$ for $k\in\mathcal{K}$. Assume 
that $q_k(t)\equiv q>0$ is constant, that $p_k(t)\equiv0$ for all $k\geq1$, and allow a single time-varying source term $p_0(t)\geq0$ at $k=0$. Then
\begin{equation}
[S_k](t)\equiv [S_k^L](t),
\end{equation}
where $[S_k^L]$ is given by~\eqref{SourceTermExplicit}.
\end{corollary}
\begin{proof}
The proof follows from~\eqref{EquationDiscrete1D2Sided_1}, the fact that $[S_k^R]\equiv1$ and that $[S_k^L]$ is given by~\eqref{SourceTermExplicit}.
\end{proof}

\section{Space-continuous limits}\label{sec:1D1S_SpaceContinuous}

In this section, we discuss the space-continuous models that can be obtained after a space rescaling in our space-discrete process. Since we have shown in \autoref{th:1D2SWithP_R} that 1D two-sided lattices reduce to 1D one-sided lattices, we focus on the latter case to make the presentation clearer. 


Equation~\eqref{EquationDiscrete1D1SidedWithP_R_Dt} can be rewritten as
{
\begin{equation}\label{EquationDiscrete1D1SidedWithP_Rewriting}
{}
\begin{aligned}
[S_k]'
&+q_ke^{-\int_0^tp_k}[S_k^0]\left([S_k]-[S_{k-1}]\right)+\left(p_k+q_k\left(1-e^{-\int_0^tp_k}[S_k^0]\right)+r_k\right)[S_k]
\\
&=q_k e^{-\int_0^tp_k}[S_k^0]\left([R_{k-1}^0]+\frac{r_k(0)}{q_k(0)}+\int_0^t \left(\frac{r_k}{q_k}\right)'(s)\frac{[S_k](s)}{[S_k^0]}ds\right).
\end{aligned}
\end{equation}
}
Since the term $[S_k]-[S_{k-1}]$ is a discrete spatial derivative, this suggests that we can derive a space-continuous PDE by a proper space rescaling.
To do that, assume that the individuals are placed on the inifinite line $\mathcal{K}=\mathbb{Z}$,
fix $0<\Delta x\ll1,$ and proceed to the space-rescaling
\begin{equation*}
[{S}](t,x):= \left[S_{k=\lfloor\frac{x}{\Delta x}\rfloor}\right](t), \qquad \forall x\in\Delta x\ \mathbb{Z}.
\end{equation*}
We can similarily define $p(t,x)=p_{k=\lfloor\frac{x}{\dx}\rfloor}(t)$, $q(t,x)=q_{k=\lfloor\frac{x}{\dx}\rfloor}(t)$, $r(t,x)=r_{k=\lfloor\frac{x}{\dx}\rfloor}(t)$, $[S^0](x)=\left[S^0_{k=\frac{x}{\Delta x}}\right]$, $[R^0](x)=\left[R^0_{k=\frac{x}{\Delta x}}\right]$.

The limit of~\eqref{EquationDiscrete1D1SidedWithP_Rewriting} as $\dx\to0$ depends on some additional assumptions on how the parameters rescale. In the present paper, we shall only discuss the formal passage to the limit and leave the rigorous proof to future works.


\subsection{Limiting {integrodifferential} ODE}\label{sec:limitingODE}
First, let us make the following assumptions:
\\
\textit{Rescaling Assumption 1.}
the parameters $p_k(t)$, $q_k(t)$, $r_k(t)$ and the initial conditions $[S^0_k]$, $[R^0_k]$ converge as $\dx\to0$ { locally uniformly in $C^1$} in time and space to some smooth functions $p(t,x)$, $ q(t,x)>0$, $r(t,x)$, $[S^0(x)]$, and $[R^0(x)]$, respectively.

Under these assumptions, we have that $[S_k]-[S_{k-1}]=O(\dx)$ and so
\begin{equation}\label{Locally_Constant}
{}[ S](t,x)-[ S](t,x-\Delta x)=O(\dx).
\end{equation}
Hence, $S(t,x;\dx)$ converges to the solution of the {(integrodifferential)} ODE
\begin{equation}\label{EquationDiscrete1D1Sided_LimitingODE}
\left\{\begin{aligned}
&\D_t [{S}](t,x)+\Big[p(t,x)+q(t,x)\big(1-[ S^0](x) e^{-\int_0^t p(\cdot,x)}\big)+r(t,x)\Big][ S]
\\
&\qquad\qquad\qquad\qquad =q(t,x) e^{-\int_0^tp(\cdot,x)}[S^0](x)
{
\left([R_{k-1}^0]+\frac{r(0,x)}{q(0,x)}+\int_0^t \D_t\left[\frac{r}{q}\right](s,x)\frac{[S](s,x)}{[S^0](x)}ds\right)
},\\
&S(0,x)=[S^0(x)].
\end{aligned}\right.\end{equation}

In particular, if $p$, $q$, and $r$ do not depend on time (but are allowed to depend on space), we can integrate the above ODE and get the explicit expression
\begin{equation}\label{SolutionLimitingODE_simple}
{}[S](t,x) =[S^0](x) e^{-(p(x)+q(x)+r(x))t + [S^0](x)q \frac{1-e^{-pt}}{p}}
\left(
1
+
(r+q[R^0](x))
\int_0^t e^{(q+r)\tau - [S^0]q \frac{1-e^{-p\tau }}{p}}d\tau
\right).
\end{equation}
Note that the solution~\eqref{SolutionLimitingODE_simple} is equal pointwise to the spatially-uniform solution~\eqref{Expression_S_Uniform_SIR}. Intuitively, this is because under Rescaling Assumptions 1, the spatial heterogeneity becomes locally uniform as $\dx\to0$ and there are infinitely many nodes between two distinct points $x,x'\in\R$. Therefore, the solution can be computed at each point as a spatially constant solution. 

The solution~\eqref{SolutionLimitingODE_simple} is thus spatially homogeneous in the short spatial scale $O(dx)$, but inhomogeneous on the long spatial scale $x=O(1)$. This multiple-scale property occurs in many situations, e.g., the slowly-varying amplitude approximation in optics.

The convergence of the solution of~\eqref{EquationDiscrete1D1SidedWithP_Rewriting} to the solution of the limiting ODE~\eqref{EquationDiscrete1D1Sided_LimitingODE} as $\dx\to0$ under \textit{Rescaling Assumptions 1} is illustrated in \autoref{fig:SpaceContinuousLimits_ODE}.

\begin{remark}
Solution~\eqref{SolutionLimitingODE_simple} can be derived if the initial condition $S^0,I^0,R^0$ and the parameters $p,q,r$ are only piecewise smooth. It allows us to include discontinuous coefficients (we can then solve the equation in each of the continuous piece separately). This may be interesting for several problems:
\begin{itemize}
\item The Riemann problem
where the initial condition is a step function $[S^0](x)=0$ when $x<0$ and $[S^0](x)=1$ when $x>0$.
\item The gap problem, when $q(x)$ and $p(x)$ are spatially compactly supported or periodic
\item The time-control (or "optimal duration campaign" problem) when $p$ and $q$ {are piecewise constant with respect to time and change value at some time $t_0>0$}.
\end{itemize}
\end{remark}

\begin{figure}
\center
\includegraphics[width=0.6\linewidth]{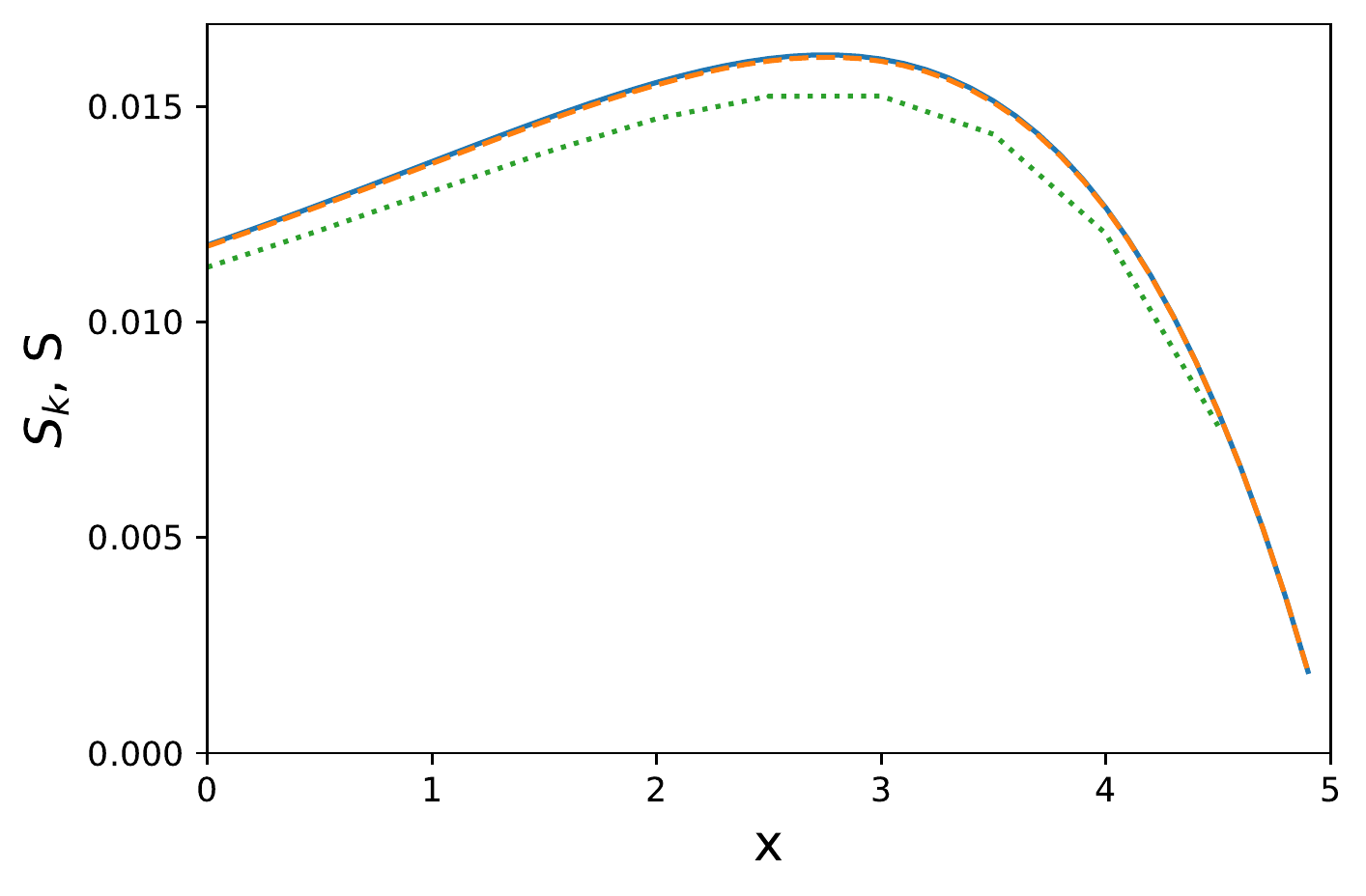}  
  \caption{Convergence to the {(integrodifferential)} ODE limit~\eqref{EquationDiscrete1D1Sided_LimitingODE} under \textit{Rescaling Assumptions 1}. Snapshot at $t=2$ of $[S_k(t)]$ on a segment $k=0,\dots, \lfloor\frac{5}{\dx}\rfloor$ so that $x\in(0,5)$.
Blue solid line represents the explicit solution~\eqref{SolutionLimitingODE_simple} of the limiting ODE~\eqref{EquationDiscrete1D1Sided_LimitingODE} ; Green dotted line and orange dashed line represent the solution $[S_k](t)$ of \eqref{EquationDiscrete1D1SidedWithP_Rewriting} for $\dx=0.5$ and $\dx=0.1$ respectively. Choice of parameters and initial conditions : $ p(x)=1-\frac{x}{5}$, $ q(x)= 5+x$, $ r(x)=2-\frac{2x}{5}$, $[ S^0](x)=0.2-\frac{x}{25}$, $[ R^0](x)=0.2+\frac{3x}{50}$}
 \label{fig:SpaceContinuousLimits_ODE}
\end{figure}

%

\subsection{Limiting PDE.}\label{sec:limitingPDE}
The former rescaling gives a limiting ODE for which space can be viewed as a parameter since there is no derivative with respect to the variable $x$. In other words, there is no spatial coupling in the system as $\dx\to0$, as expressed in~\eqref{Locally_Constant}. Intuitively, this is because the influence of any node on its right neighbors propagates at the speed $c(t,x):=[S^0]qe^{-\int_0^t p}\dx$, see~\eqref{EquationDiscrete1D1SidedWithP_Rewriting}. Thus, $c(t,x)\to0$ as $\dx\to0$ and so the limiting equation is spatially decoupled. 
Therefore, in order to get a limiting equation that includes space derivative, one has to rescale $q$ by a factor of $1/\dx$, namely, to assume that $q=\frac{\tilde{q}}{\dx}$ for a function $\tilde q$ independent of $\dx$ (or which converges locally uniformly as $\dx\to0$). This way, the term $q\left([ S](t,x)-[ S](t,x-\Delta x)\right)$ is of order $1$ and converges to $\tilde q\nabla_x [S]$ as $\dx\to0$.

If we assume, however, that the contagion coefficient becomes large as the space-step $\dx$ vanishes, i.e., that $q$ is of order $1/\dx$, we also need to rescale the other parameters. Indeed, we see from equation~\eqref{EquationDiscrete1D1SidedWithP_Rewriting} that, in order not to get the singular limit $S(t,x)\equiv0$ as $\dx\to0$ and $q=O(1/\dx)$, we need the term $1-[S^0] e^{-pt}$ to be of order $O(\dx)$ as $\dx\to0$.  To fullfill thoses conditions, we make the following
\\
\textit{Rescaling Assumption 2.}
{ Assume that $q_k(t)>0$ and that the ratio $\frac{r_k(t)}{q_k(t)}$ does not depend on time.} Further assume that the parameters and initial conditions scale as
\begin{equation}\label{Assumptions2}
q=\frac{\tilde q(t,x;\dx)}{\dx}>0,
\quad
p= \tilde p(t,x;\dx) \dx,
\quad 
r=\tilde r,
\quad
[I^0(x)]=[\tilde I^0(x;\dx)]\dx,
\quad
[R^0(x)]=[\tilde R^0(x;\dx)]\dx,
\end{equation}
where the ``tilde functions'' are assumed to converge locally uniformly as $\dx\to0$.
%


\paragraph*{} 
Since we assume that $\frac{r(t,x)}{q(t,x)}$ does not depend on time, equation~\eqref{EquationDiscrete1D1SidedWithP_Rewriting} boils down to
{
\begin{equation}\label{EquationDiscrete1D1SidedWithP_Rewriting_BIS}
{}
[S_k]'
+q_ke^{-\int_0^tp_k}[S_k^0]\left([S_k]-[S_{k-1}]\right)
+\left(p_k+q_k\left(1-e^{-\int_0^tp_k}[S_k^0]\right)+r_k\right)[S_k]
 = e^{-\int_0^tp_k}[S_k^0]\left(q_k[R_{k-1}^0]+r_k\right).
\end{equation}
}
Under Rescaling Assumptions~2, the terms in~\eqref{EquationDiscrete1D1SidedWithP_Rewriting_BIS} have the following limit for any fixed $t\geq0$ as $\dx\to0$:
\begin{gather*}
[S^0]\sim1\\
 q [S^0] e^{-\int_0^tp}\big([S](t,x)-[S](t,x-\dx)\big)
 \sim
 \tilde q \D_x [S](t,x),
 \\
p+q\big(1- e^{-\int_0^tp}[S^0]\big)+r 
\sim 
\tilde{q}([\tilde I^0(x)]+[\tilde R^0(x)]+\int_0^t \tilde p)+\tilde{r},
\\
 e^{-\int_0^tp}[S^0]\big(q[R^0]+r\big)
\sim
\tilde q[\tilde R^0]+\tilde r,
\end{gather*}
and so the limiting equation satisfied by $[S(t,x)]$ is
\begin{equation}
\label{Limiting_PDE}\left\{\begin{aligned}
&\D_t [S](t,x)+\tilde{q}(t,x)\D_x [S]+\left[\tilde{q}\left([\tilde I^0](x)+[\tilde R^0](x)+\int_0^t \tilde p(t,x)\right)+\tilde r(t,x)\right][S]=\tilde q[\tilde R^0]+\tilde r,\\
&[S](0,x)=1.
\end{aligned}\right.\end{equation}

The convergence of the solution of~\eqref{EquationDiscrete1D1SidedWithP_Rewriting_BIS} to the solution of the limiting PDE~\eqref{Limiting_PDE} as $\dx\to0$ under \textit{Rescaling Assumptions 2} is illustrated in \autoref{fig:SpaceContinuousLimits_PDE}. {The PDE~\eqref{Limiting_PDE} is solved numerically using a standard explicit finite-difference Euler scheme.}

Equation~\eqref{Limiting_PDE} can be solved using the method of the caracteristics. For example, assuming for simplicity that there is no recovery (i.e., $[\tilde R^0]\equiv \tilde r\equiv 0$), no source term (i.e., $p\equiv 0$) and that $\tilde q(t,x)\equiv \tilde q$ is contant, the solution reads
\begin{equation}\label{PDE_SolutionExplicit_SI}
S(t,x)= e^{-\int_{x-\tilde q t}^x [\tilde I^0](s)ds}.
\end{equation}

\begin{figure}
\center
\includegraphics[width=0.6\linewidth]{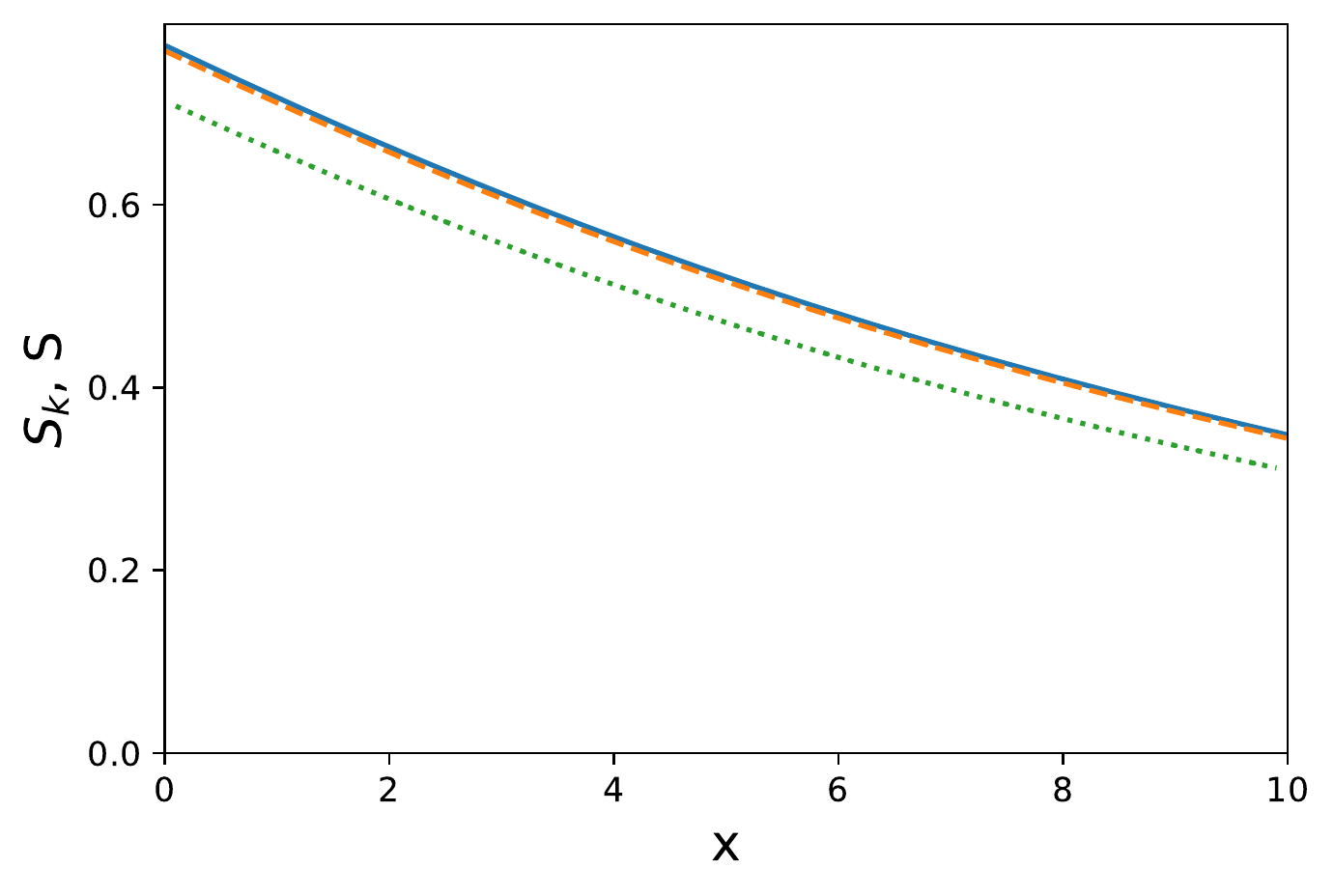}  
  \caption{Convergence towards PDE limit~\eqref{Limiting_PDE} under \textit{Rescaling Assumptions 2}.
Snapshot at $t=2$ of $[S_k(t)]$ on a segment $k=0,\dots, \lfloor\frac{10}{\dx}\rfloor$, with $\dx>0$, so that $x\in(0,10)$. 
Blue solid line represents the solution $S(t,x)$ of the limiting PDE~\eqref{Limiting_PDE} ; Green dotted line and orange dashed line represent the solution $[S_k](t)$ of  \eqref{EquationDiscrete1D1SidedWithP_Rewriting_BIS} for $\dx=0.1$ and $\dx=0.01$ respectively. Choice of parameters and initial conditions :  $\tilde p(x)=0.1+\frac{0.2x}{10}$, $\tilde q(x)=1+\frac{x}{10}$, $\tilde r(x)=0.3+\frac{0.5x}{10}$, $[\tilde I^0](x)=0.2+\frac{0.5x}{10}$, $[\tilde R^0](x)=0.5-\frac{0.3x}{10}$. The parameters $p$, $q$, $r$ and initial conditions $[S^0]$, $[I^0]$, $[R^0]$ are then defined by the rescaling~\eqref{Assumptions2}.}
 \label{fig:SpaceContinuousLimits_PDE}
\end{figure}

\subsection{Discussion}
The two space-continuous models proposed above are obtained by different rescalings, each of which corresponding to a different modeling situation:
\begin{itemize}
\item The {(integrodifferential)} ODE limit~\eqref{EquationDiscrete1D1Sided_LimitingODE} discussed in Section~\ref{sec:limitingODE} corresponds to the case where the contagion between two consecutive nodes is driven by their discrete distance on the graph rather than by their metric distance on the line. In this case, the propagation of the epidemic propagates at finite speed $q$ on the graph.
When the rescaling factor $\dx$ goes to $0$, any interval on the real line contains infinitely many nodes, therefore the epidemics do not propagate on the line in finite time. This explains why the obtained limiting {(integrodifferential)} ODE~\eqref{EquationDiscrete1D1Sided_LimitingODE} is spatially decoupled. The {(integrodifferential)} ODE limit~\eqref{EquationDiscrete1D1Sided_LimitingODE} is relevant from the modeling point of view when the contagion phenomenon occurs at a small space-scale.
\item In contrast, the PDE limit~\eqref{Limiting_PDE} obtained in Section~\ref{sec:limitingPDE} corresponds to a contagion phenomenon that is driven by the metric distance between nodes rather than on their discrete distance on the graph. In this case, the propagation of the epidemic occurs at a finite speed $\tilde q$ on the line, as expressed by the explicit form of the solution in~\eqref{PDE_SolutionExplicit_SI}. It is then natural to consider the parameter $p$ and initial conditions $[I^0]$ and $[R^0]$ as space-densities, as suggested by the Rescaling Assumptions~\eqref{Assumptions2}. The PDE limit~\eqref{Limiting_PDE} has the advantage to qualitatively reproduce the spatial propagation at finite speed occurring in the discrete model~\eqref{EquationDiscrete1D1SidedWithP_Rewriting}.
This PDE rescaling is also relevant when the contagion rate outweighs the source term (i.e., $p\ll q$).
\end{itemize}


%


\section{Final remarks}\label{sec:conclusion}

In the present paper, we establish an \emph{exact deterministic description} of stochastic SIR-Bass epidemics on 1D lattices by showing that the probability of infection at a given point in space and time can be obtained as the solution of a deterministic ODE system on the lattice. Our framework allows any type of heterogeneity on the parameters and initial conditions. In addition, our results precisely describe the spatio-temporal dynamics at a local scale, which is still a major challenge in epidemiology~\cite{Pellis2015b}.




Our results focus on 1D lattices (i.e., finite, semi-infinite, and infinite lines) where the infection only occurs from a node to its direct neighbors. This is the first step towards more realistic contact networks, such as 2D lattices, scale-free networks, small-world networks, as well as other types of networks that are relevant from the application point of view in the context of epidemiology or marketing.
Most likely, an exact deterministic description such as the one obtained in this paper cannot be obtained for 2D lattices or other contact networks\footnote{When the underlying undirected graph contain loops, such as in $n$D lattices for $n\geq2$, it is not possible even for the Bass model to find a finite system of equations on the marginals as in~\autoref{th:1D1SWithP}, see~\cite{Sharkey2015}.}.
We believe however that the analytical tools introduced in the present paper will be useful for further analysis of epidemiological models in complex networks under the pair-approximation assumption~\cite{House}.

\paragraph{Acknowledgement.}
The authors are deeply thankful to Professor Steven Schochet and Professor Eitan Tadmor for useful discussions.

\bibliographystyle{abbrv}
\bibliography{library}

\appendix
\section{Appendix.}

\subsection{Proof of~\autoref{th:coro_patient_zero}}\label{app:proof_line}
Expression~\eqref{Explicit_withP_Uk(t)} can be verified by direct substitution in~\eqref{EquationDiscrete1D1SidedWithP}. 
To show~\eqref{Explicit_Esperance_1D1S}, we first notice that  $\mathbb{E}[N_K^n]= \sum_{k=1}^{K}[I_k](t)$. Using $[I_k](t)=1-[S_k](t)$ and expression~\eqref{Explicit_withP_Uk(t)}, we compute
\begin{align*}
\mathbb{E}[N_K^n]
&=\sum_{k=1}^{K}\left(1-e^{-(p+q)t}\sum\limits_{l=0}^{k-1} \frac{\left(q \frac{1-e^{-pt}}{p}\right)^l}{l!}\right)\\
&=K-e^{-(p+q)t} \sum_{k=1}^{K}\sum\limits_{l=0}^{k-1} \frac{\left(q \frac{1-e^{-pt}}{p}\right)^l}{l!}\\
&=K-e^{-(p+q)t} \sum\limits_{l=0}^{K-1} \sum_{k=l+1}^{K}\frac{\left(q \frac{1-e^{-pt}}{p}\right)^l}{l!}\\
&=K-e^{-(p+q)t} \sum\limits_{l=0}^{K-1} (K-l)\frac{\left(q \frac{1-e^{-pt}}{p}\right)^l}{l!}.
\end{align*}
We obtain~\eqref{Explicit_Esperance_1D1S} from the change of variable $l'=K-l$ in the above expression.

To show~\eqref{ExpectedFraction}, note that above expression gives
\begin{align*}
\frac{\mathbb{E}[N_K^n]}{K}=1-e^{-(p+q)t} \sum\limits_{l=0}^{K-1} \frac{\left(q \frac{1-e^{-pt}}{p}\right)^l}{l!}+\frac{1}{K}\sum\limits_{l=0}^{K-1} l\frac{\left(q \frac{1-e^{-pt}}{p}\right)^l}{l!}.
\end{align*}
Then, we conclude that, as $K\to+\infty$,
$$\sum\limits_{l=0}^{K-1} \frac{\left(q \frac{1-e^{-pt}}{p}\right)^l}{l!}\sim e^{q \frac{1-e^{-pt}}{p}},$$
and that
\begin{align*}
\frac{1}{K}\sum\limits_{l=0}^{K-1} l\frac{\left(q \frac{1-e^{-pt}}{p}\right)^l}{l!}
= \frac{q \frac{1-e^{-pt}}{p}}{K}\underbrace{\sum\limits_{l=1}^{K-1} \frac{\left(q \frac{1-e^{-pt}}{p}\right)^{l-1}}{(l-1)!}}_{=O(1)}\to 0.
\end{align*}

\subsection{Proof of~\autoref{th:coro_patient_zero_SIR}}\label{sec:Proof_of_th:coro_patient_zero_SIR}

A straightforward substitution shows that expression~\eqref{Explicit_SIR_Sk(t)} solves equation~\eqref{EquationDiscrete1D1SidedWithP_R_Dt} and satisfies the initial conditions~\eqref{Initial_Condition_PatientZero_1D1S_General}.

To prove~\eqref{Explicit_SIR_Ek}, 
first note that $\mathbb{E}[N_K^n]= \sum_{k=1}^{K}[S_k(t)]$, and so, using~\eqref{Explicit_SIR_Sk(t)}, we find
\begin{align*}
\mathbb{E}[N_K^n]
&=\sum_{k=1}^{K}\left(1-\left(\frac{q}{q+r}\right)^k\left(1- e^{-(q+r)t}\sum\limits_{l=0}^{k-1} \frac{\left((q+r) t\right)^l}{l!}\right)\right)\\
&=K- \frac{q}{q+r}\frac{1-\left(\frac{q}{q+r}\right)^K}{1-\frac{q}{q+r}}- e^{-(q+r)t} \sum\limits_{k=1}^{K}\sum\limits_{l=0}^{k-1} \left(\frac{q}{q+r}\right)^k\frac{\left((q+r) t\right)^l}{l!}\\
&=K- \frac{q}{r}\left(1-\left(\frac{q}{q+r}\right)^K\right)- e^{-(q+r)t} \sum\limits_{l=0}^{K-1}\frac{\left((q+r) t\right)^l}{l!} \left(\frac{q}{q+r}\right)^l\frac{1-\left(\frac{q}{q+r}\right)^{K-L+1}}{1-\frac{q}{q+r}}.
\end{align*}
Then, we get~\eqref{Explicit_SIR_Ek} from a change of variable $l'=K-l$ in the above expression.

\subsection{Spatial Markovian property - proof of~\autoref{rmk:Markov_1D2S_Proof}}\label{sec:Appendix_SpatialMarkov}
The proof of~\autoref{th:1D2SWithP_R} relies on the spatial Markovian property stated in~\autoref{rmk:Markov_1D2S}. This section is devoted to the proof of this result. For the sake of generality, we actually prove the following more general statement.
\begin{lemma}\label{rmk:Markov_1D2S}
Assume that the individuals are placed on a 1D two-sided lattice (see~\eqref{Structure_Assumption_1D2S}) and that the initial conditions are uncorrelated, see~\eqref{hyp:initial_cond_uncor}.

For any integer $k\in\mathbb{Z}$, denote by $X^n_{k-}:=(x^n_l)_{l<k}$ and $X^n_{k+}:=(x^n_l)_{l>k}$ the left and right neighbors of $k$. We have that
\begin{equation}\label{eq:Markov_1D2S}
x^n_{k-1}\independent x^n_{k+1}\vert (x_k^m)_{0\leq m<n},\qquad \forall k\in\mathbb{Z},
\end{equation}
meaning that the state of the left and right neighbors of $k$ are independent under the conditioning with respect to the state of $k$ for all preceeding times. In particular,
\begin{equation}\label{Markov_1D2S_1}
x_{k-1}^n\independent x_{k+1}^n\vert S_k^n.
\end{equation}
\end{lemma}
\begin{proof}
Let $(\tilde x_{k-1}^n,\tilde x_{k+1}^n)\in\{s,i,r\}^2$ be a possible realization of $(x_{k-1}^n,x_{k+1}^n)$, and let us denote by $Y_-^n:=\left\{ x_{k-1}^n=\tilde x_{k-1}^n\right\}$ and $Y_+^n:=\left\{ x_{k+1}^n=\tilde x_{k+1}^n\right\}$ the events that $x_{k+1}^n$ and $x_{k-1}^n$ equal this realization respectively. For clarity, we also dentote by $X_k^n:=(x_k^m)_{0\leq m<n}$.

Let us prove by induction on $n\geq0$ that
\begin{equation}\label{1D2S_Markov_InductionAssumption}
\P\(Y_-^n\cap Y_+^n\vert X_k^n\)= \P\(Y_-^n\vert X_k^n\)\P\(Y_+^n\vert X_k^n\).
\end{equation}
By assumption on the initial conditions, this property holds for $n=0$. Assume that the above property holds for some $n\geq 0$.
Then, conditionning with respect to $Y_-^n\cap Y_+^n$ and using~\eqref{RelationClasssique}, we have
\begin{equation}\label{1D2S_Markov_Proof1}
\P\(Y_-^{n+1}\cap Y_+^{n+1}\vert X_k^{n+1}\)= \P\(Y_-^{n+1}\cap Y_+^{n+1}\vert X_k^{n+1},Y_-^{n}\cap Y_+^{n}\)\P\(Y_-^{n}\cap Y_+^{n}\vert X_k^{n+1}\)
\end{equation}
From the assumption that the $(x_k^{n+1})_{k\in\mathcal{K}}$ are all independent if we condition by $(x_k^n)_{k\in\mathcal{K}}$, we deduce
\begin{equation*}
\P\(Y_-^{n+1}\cap Y_+^{n+1}\vert X_k^{n+1},Y_-^{n}\cap Y_+^{n}\)=\P\(Y_-^{n+1}\vert X_k^{n+1},Y_-^{n}\cap Y_+^{n}\)\P\(Y_+^{n+1}\vert X_k^{n+1},Y_-^{n}\cap Y_+^{n}\).
\end{equation*}
Moreover, from the assumption that $\mathcal{K}$ is a 1D two-side lattice, we further deduce that 
\begin{equation}\label{1D2S_Markov_Proof2}
\P\(Y_-^{n+1}\cap Y_+^{n+1}\vert X_k^{n+1},Y_-^{n}\cap Y_+^{n}\)=\P\(Y_-^{n+1}\vert X_k^{n+1},Y_-^{n}\)\P\(Y_+^{n+1}\vert X_k^{n+1}, Y_+^{n}\).
\end{equation}
Besides, the induction assumption~\eqref{1D2S_Markov_InductionAssumption} implies that
\begin{equation}\label{1D2S_Markov_Proof3}
\P\(Y_-^{n}\cap Y_+^{n}\vert X_k^{n+1}\)= \P\(Y_-^{n}\vert X_k^{n+1}\)\P\( Y_+^{n}\vert X_k^{n+1}\).
\end{equation}
Injecting \eqref{1D2S_Markov_Proof2} and \eqref{1D2S_Markov_Proof3} into \eqref{1D2S_Markov_Proof1}, we derive
\begin{align*}
\P\(Y_-^{n+1}\cap Y_+^{n+1}\vert X_k^{n+1}\)
&= \P\(Y_-^{n+1}\vert X_k^{n+1},Y_-^{n}\)\P\(Y_-^{n}\vert X_k^{n+1}\)\P\(Y_+^{n+1}\vert X_k^{n+1}, Y_+^{n}\)\P\( Y_+^{n}\vert X_k^{n+1}\)\\
&= \P\(Y_-^{n+1}\vert X_k^{n+1}\)\P\(Y_+^{n+1}\vert X_k^{n+1}\),
\end{align*}
which proves that~\eqref{1D2S_Markov_InductionAssumption} holds at the rank $n+1$.

We have thus proved that~\eqref{1D2S_Markov_InductionAssumption} holds for all $n\geq0$, from which we immediately deduce~\eqref{eq:Markov_1D2S}.

 Property~\eqref{Markov_1D2S_1} is deduced as a particular instance of~\eqref{eq:Markov_1D2S}.
\end{proof}

\begin{remark}
In constrast with \autoref{rmk:Markov_1D2S}, we have that
\begin{equation*}
x_{k-1}^n\not \independent x_{k+1}^n\vert {I_k^n}.
\end{equation*}
The reason is that the condition $I_k^n$ only prescribes the state of $x_k^n$ at time $t^n$ and not the states $(x_k^m)_{0\leq m<n}$ at the preceeding times since it does not indicate at what time $x_k$ becomes infected.
%

\end{remark}

\end{document}